\documentclass[a4paper]{article}

\usepackage[fleqn,leqno]{amsmath}
\usepackage{amsthm}
\usepackage{amssymb}
\usepackage{enumerate}
\usepackage{booktabs}
\usepackage{rotating}
\usepackage{verbatim}
\usepackage[numbers]{natbib}
\usepackage{subcaption}
\usepackage{graphicx}
\usepackage{float}
\usepackage{hyphenat}
\usepackage{appendix}
\usepackage[affil-it]{authblk}
\allowdisplaybreaks

\title{Minimum Weight Resolving Sets of Grid Graphs}

\author{Patrick Andersen%
  \thanks{Email: \texttt{pat.j.andersen@gmail.com}; Phone: \texttt{+61-412579238}; Corresponding author}
}
\author{Cyriac Grigorious %
	}
\author{Mirka Miller %
	}
\affil{School of Mathematical and Physical Sciences, The University of Newcastle, Australia}

\newtheorem{lemma}{Lemma}
\newtheorem{theorem}{Theorem}
\newtheorem{proposition}{Proposition}
\newtheorem{corollary}{Corollary}

\makeatletter
\def\@maketitle{%
  \newpage
  \null
  \vskip 2em%
  \begin{center}%
  \let \footnote \thanks
    {\Large\bfseries \@title \par}%
    \vskip 1.5em%
    {\normalsize
      \lineskip .5em%
      \begin{tabular}[t]{c}%
        \@author
      \end{tabular}\par}%
    \vskip 1em%
    {\normalsize \@date}%
  \end{center}%
  \par
  \vskip 1.5em}
\makeatother

\begin{document}
\maketitle
\begin{abstract}
For a simple graph $G=(V,E)$ and for a pair of vertices $u,v \in V$, we say that a vertex $w \in V$ resolves $u$ and $v$ if the shortest path from $w$ to $u$ is of a different length than the shortest path from $w$ to $v$. A set of vertices ${R \subseteq V}$ is a resolving set if for every pair of vertices $u$ and $v$ in $G$, there exists a vertex $w \in R$ that resolves $u$ and $v$. The minimum weight resolving set problem is to find a resolving set $M$ for a weighted graph $G$ such that$\sum_{v \in M} w(v)$ is minimum, where $w(v)$ is the weight of vertex $v$. In this paper, we explore the possible solutions of this problem for grid graphs $P_n \square P_m$ where $3\leq n \leq m$. We give a complete characterisation of solutions whose cardinalities are 2 or 3, and show that the maximum cardinality of a solution is $2n-2$. We also provide a characterisation of a class of minimals whose cardinalities range from $4$ to $2n-2$. 
\end{abstract}

\section{Introduction}
Let $G=(V,E)$ be a simple graph, and for each pair of vertices $u,v \in V$, let $d(u,v)$ denote the length of the shortest path from $u$ to $v$, where ${d(u,u)=0}$   ${\forall u \in V}$ and $d(u,v) = \infty$ if $u$ and $v$ are disconnected. For two distinct vertices $u,v \in V$, a vertex $w$ is said to \textit{resolve} $u$ and $v$ if $d(w,u) \neq d(w,v)$. A set of vertices ${R \subseteq V}$ is said to be a \textit{resolving set} if for every pair of vertices $u$ and $v$ in $G$, there exists a vertex $w \in R$ that resolves $u$ and $v$. The elements of a resolving set are often called $landmarks$. For a graph $G$, a \textit{metric basis} is a resolving set of minimum cardinality, and the cardinality of a metric basis is the \textit{metric dimension} of $G$. Applications of metric bases and resolving sets arise in various settings such as network optimisation \cite{networkdiscovery}, chemistry and drug discovery \cite{chartrand2000}, robot navigation \cite{khuller}, digitisation of images \cite{melter}, and solutions to the Mastermind game \cite{chvatal}.\\
The problem of finding the metric dimension of a graph was introduced independently by Harary and Melter \cite{harary}, and Slater \cite{slater} and has been widely investigated in combinatorics literature. Khuller, Raghavachari, and Rosenfeld \cite{khuller} showed that the problem of finding the metric dimension is NP-hard for general graphs and developed a $(2 \text{ln}(n) + O(1))$ approximation algorithm. They also showed that the metric dimension of a graph is 1 iff the graph is a path and they showed that the problem is polynomial-time solvable for the case of trees. Beerliova et al. \cite{networkdiscovery} showed that no $o(\text{log}(n))$ approximation algorithm exists if $P \neq NP$. Chartrand et al. \cite{chartrand2000} proved that the only graph whose metric dimension is $|V|-1$ is $K_{|V|}$ and characterised the graphs whose metric dimension is $|V|-2$. Melter and Tomescu \cite{melter} proved that the metric dimension of grid graphs $P_n \square P_m$ is 2 and that metric bases correspond to two endpoints of a boundary edge of the grid. For more results on the metric dimensions of graphs, we refer the reader to \cite{families} and \cite{survey}.\\
We now consider a generalisation of the metric dimension problem that was first introduced by Epstein, Levin and Woeginger \cite{weightedmd} where we have a given assignment of positive weights $w(v)$ to each vertex $v \in V$. The problem is to find a \textit{minimum weight resolving set} $M \subseteq V$ such that the sum of the weights of the vertices in $M$, $\sum_{v \in M} w(v)$, is minimum. We refer to this problem as the minimum weight resolving set problem. Epstein, Levin and Woeginger showed that this problem is NP-hard for general graphs and found that the only possible solutions to the minimum weight resolving set problem correspond to \textit{minimal resolving sets} that are minimal with respect to inclusion, i.e. resolving sets $R$ where $\nexists v \in R$ such that $R-\{v\}$ is resolving. The same authors developed polynomial time algorithms for paths, trees, cycles, wheels and $k$-augmented trees (trees with an additional $k$ edges) by exhaustively enumerating the minimal resolving sets for these graphs and choosing the one with the minimum weight. As far as we are aware, these are the only graphs for which the minimum weight resolving set problem has previously been explored in the literature.\\
\textbf{Our Results}. Following the work of Epstein, Levin and Woeginger, we explore the minimum weight resolving set problem for grid graphs, $P_n \square P_m$, where $3 \leq n \leq m$. We completely characterise the minimal resolving sets of cardinality 2 and 3 for these graphs and find that for all minimal resolving sets $M$ for the grid, $2\leq |M| \leq 2n-2$. We also give a characterisation of a class of minimals whose cardinalities range between $3$ and $2n-2$ and provide a weak characterisation of a resolving set for the grid.

\section{Terminology}
Given the graph $P_n \square P_m$ where $m,n \geq 3$, if we label the vertices of $P_n$ by $u_0,u_1,\dots ,u_{n-1}$ and the vertices of $P_m$ by $v_0,v_1,\dots ,v_{m-1}$, then we have the natural labelling of the vertices of $P_n \square P_m$ where each vertex is labelled with $(u_i,v_j), i \in \{0,1,\dots n-1\} ,j \in \{0,1,\dots m-1\}$. This labelling has an obvious connection to the coordinates on the Cartesian plane, hence without loss of generality, we will refer to the first coordinate of the label as the coordinate of the vertex in the horizontal direction, and the second coordinate as the coordinate of the vertex in the vertical direction. For simplicity, we will refer to the vertex labelled by $(u_i,v_j)$ as vertex $(i,j)$. It is clear that the shortest distance between two vertices $(i,j)$ and $(k,l)$ is the Manhattan distance between the coordinates, or $|i-k| + |j-l|$. 
\\
All vertices in grid graphs have a degree of 2, 3 or 4. We give terms for each of these vertex types:
\begin{itemize}
\item Vertices of degree 2 are \textit{corner vertices}.
\item Vertices of degree 3 are \textit{side vertices}.
\item Vertices of degree 4 are \textit{interior vertices}.
\end{itemize}
The vertices of degree 2 and 3 are also known as \textit{boundary vertices} and the set of all boundary vertices is referred to as the boundary of the grid.\\
A \textit{line} in the grid is a path of length $n$ or $m$ in which either every vertex of the path has the same horizontal coordinate (a horizontal line), or every vertex of the path has the same vertical coordinate (a vertical line).\\ 
A \textit{side} of the grid is a line which contains only side vertices, except for the two endpoints of the line which are corner vertices. Two sides are adjacent if they share an endpoint and are opposite otherwise.\\
From this point onwards, we will refer to a minimal resolving set as a \textit{minimal} and we refer to a minimal of cardinality $k$ as a \textit{$k$-minimal}.
\section{Results}
We start with the characterisation of the metric bases of the grid, i.e. the $2$-minimals, given by Melter and Tomescu \cite{melter}.

\begin{theorem}[\cite{melter}]
A set $M$ of cardinality 2 is a $2$-minimal if and only if it contains two corners that share a side.
\end{theorem}

\begin{figure}[H]
        \centering
        \includegraphics[width=0.35\textwidth]{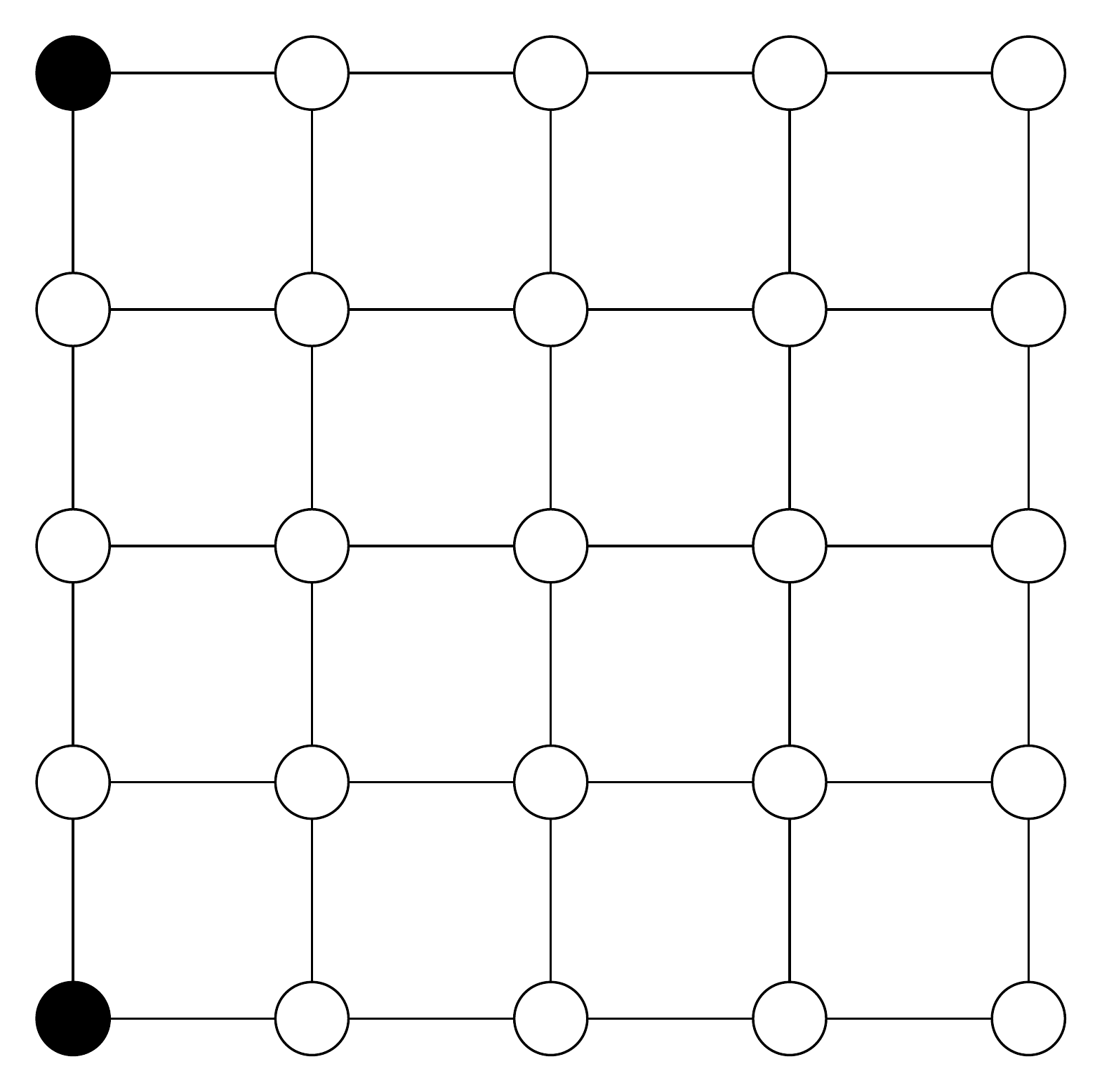}
        \caption{An example of a metric basis where the basis elements are in black.}
\end{figure}

We now attempt to characterise all the $3$-minimals. The previous theorem gives us an important property of all minimals whose cardinality is greater than 2.

\begin{proposition} \label{cornerprop}
All $k$-minimals, where $k \geq 3$, do not contain more than one corner vertex. 
\end{proposition}
\begin{proof}
Suppose we had a minimal $M$ such that $|M| \geq 3$. If $M$ contains two corner vertices on the same side, then a metric basis $B$ is a proper subset of $M$. Since no minimal is the proper subset of another minimal and $B$ is a 2-minimal, this leads to a contradiction.\\
Now suppose $M$ contains two corners $u,v$ that are not on the same side (opposite corners). Since $v$ is the only vertex that has distance $(n-1)+(m-1)$ from $u$, $v$ is resolved by $u$. Furthermore, since the grid is symmetric about its diagonals, every pair of vertices that are not resolved by $u$ will not be resolved by $v$ either. Hence if $M$ is a resolving set, then $M-v$ is also a resolving set, which implies that $M$ is not a minimal.
\end{proof}

This condition is necessary but not sufficient for a $3$-minimal. In order to see this, consider the following lemma and its corollary:

\begin{lemma}
If two vertices $u,v$ are not on the same line, then there exist two shortest paths from $u$ to $v$ of the form $u,\dots ,w_1,v$ and $u,\dots ,w_2,v$ where $w_1~\neq~w_2$.
\end{lemma}

\begin{proof}
If $u$ and $v$ are not on the same line then they differ in both horizontal and vertical position.\\
In one possible shortest path, we traverse the horizontal line from $u$ to a vertex $q_1$ which has the same horizontal position as $v$, and then traverse the vertical line from $q_1$ to $v$. In this case, the second last vertex of this path, $w_1$, will have the same horizontal coordinate as $v$ but a different vertical coordinate.\\
In another possible shortest path, we traverse the vertical line from $u$ to a vertex $q_2$ which has the same vertical position as $v$, and then traverse the horizontal line from $q_2$ to $v$. In this case, the second last vertex of this path, $w_2$, will have the same vertical coordinate as $v$ but a different horizontal coordinate.\\
It is therefore clear that $w_1 \neq w_2$.
\end{proof}

\begin{corollary} \label{c_1}
If two vertices $u,v$ are not on the same line, then there are two neighbours of $v$ that are not resolved by $u$.
\end{corollary}

If we consider a set $S$ of vertices that contains a corner and its two neighbours, then this will satisfy the conditions of Proposition~\ref{cornerprop}, however $S$ is clearly not resolving since the opposite corner of the one in $S$ is not on the same line as any of the vertices in $S$. We can use Corollary~\ref{c_1} to get another property of the $k$-minmals where $k \geq 3$.

\begin{proposition} \label{opprop}
All minimals must contain two boundary vertices on opposite sides.
\end{proposition}

\begin{proof}
Suppose we have a set $M \subset V$ such that $M$ does not contain any boundary vertices. This implies that $M$ can only contain interior vertices. Since no interior vertices are on the same line as a corner vertex, by Corollary \ref{c_1}, for all interior vertices $u$ and a particular corner $c$, there exist two neighbours of $c$, say $w_1$ and $w_2$, that are not resolved by $u$. Since $c$ only has two neighbours, the pair \{$w_1$,$w_2$\} is the same for all $u$, hence this pair of vertices is not resolved by any interior vertex and thus $M$ is not a resolving set.\\
If we add a single boundary vertex $b$ to $M$, then there is at least one corner that is not on the same line as $b$. Hence $M \cup \{b\}$ is not a resolving set.\\
If we add two boundary vertices $b_1,b_2$ to $M$ where $b_1$ and $b_2$ are not on opposite sides, then there are three possibilities:
\begin{enumerate}[(i)]
\item $b_1$ and $b_2$ are two side vertices on the same side.
\item $b_1$ and $b_2$ are two side vertices on adjacent sides.
\item One of $b_1$ and $b_2$ is a side vertex and the other is a corner vertex on the same side.
\end{enumerate}
In all three possibilities, there is still at least one corner that is not on the same line as either $b_1$ or $b_2$. Hence $M \cup \{b_1,b_2\}$ is not a resolving set.\\
Therefore, all resolving sets for the grid must contain two boundary vertices on opposite sides.
\end{proof}

There is one last property of $3$-minimals that we need in order to get a characterisation. In order to arrive at this property we need the following lemmas.

\begin{lemma} \label{quadlem}
Suppose we have a set of vertices $\{ (x_1,y_1),(x_2,y_2),\dots ,(x_k,y_k) \}$ and another vertex $(p,q)$ such that:
\begin{center}
$p<x_i \quad \forall i \in \{1,2, \dots ,k\} \quad \text{or} \quad p>x_i \quad \forall i \in \{1,2, \dots ,k\}$\\
$\quad \text{and} \quad$\\
$q<y_i \quad \forall i \in \{1,2, \dots ,k\} \quad \text{or} \quad q>y_i \quad \forall i \in \{1,2, \dots ,k\}.$
\end{center}
Then there exist two neighbours of $(p,q)$, denoted by $(p^*,q)$ and $(p,q^*)$, that are not resolved by any of the vertices $(p,q),(x_1,y_1),(x_2,y_2),\dots ,(x_k,y_k)$.
\end{lemma}

\begin{proof}
Since no vertex in $\{ (x_1,y_1),(x_2,y_2),\dots ,(x_k,y_k) \}$ is on the same line as $(p,q)$, then by Corollary~\ref{c_1}, for each $(x_i,y_i) \in \{ (x_1,y_1),(x_2,y_2),\dots ,(x_k,y_k) \}$ there are two neighbours of $(p,q)$, a horizontal neighbour $(p^*_i,q)$ and a vertical neighbour $(p,q^*_i)$, that are not resolved by $(x_i,y_i)$.\\ 
However since either $p~<~x_i$,  $\forall~i~\in~\{1,2, \dots ,k\}$, or $p~>~x_i$,  $\forall~i~\in~\{1,2, \dots ,k\}$,
\begin{center}
$(p^*_1,q)~=~(p^*_2,q)~=~\dots~=~(p^*_k,q)~=~(p^*,q)$.\\
\end{center} 
And since either $q~<~y_i$,  $\forall~i~\in~\{1,2, \dots ,k\}$, or $q~>~y_i$,  $\forall~i~\in~\{1,2, \dots ,k\}$, 
\begin{center}
$(p,q^*_1)~=~(p,q^*_2)~=~\dots~=~(p,q^*_k)~=~(p,q^*)$.\\
\end{center}
Therefore, $(p^*,q)$ and $(p,q^*)$ are not resolved by any
$(x_i,y_i) \in \{ (x_1,y_1), (x_2,y_2),\\ \dots , (x_k,y_k) \}$. And since $(p,q)$ does not resolve any of its neighbours, $(p^*,q)$ and $(p,q^*)$ are not resolved by $(p,q)$.
\end{proof}

The statement in Lemma~\ref{quadlem} is equivalent to saying that if we make $(p,q)$ the origin of coordinate axes with the $x$-axis being the line $y = q$ and the $y$-axis being the line $x = p$, then if all the vertices $(x_1,y_1),(x_2,y_2),\dots ,(x_k,y_k)$ are in the same quadrant with respect to $(p,q)$ as the origin, then the set\\
$\{ (x_1,y_1),(x_2,y_2),\dots ,(x_k,y_k),(p,q) \}$ is not resolving.\\

\begin{figure}[H]
        \centering
        \includegraphics[width=0.5\textwidth]{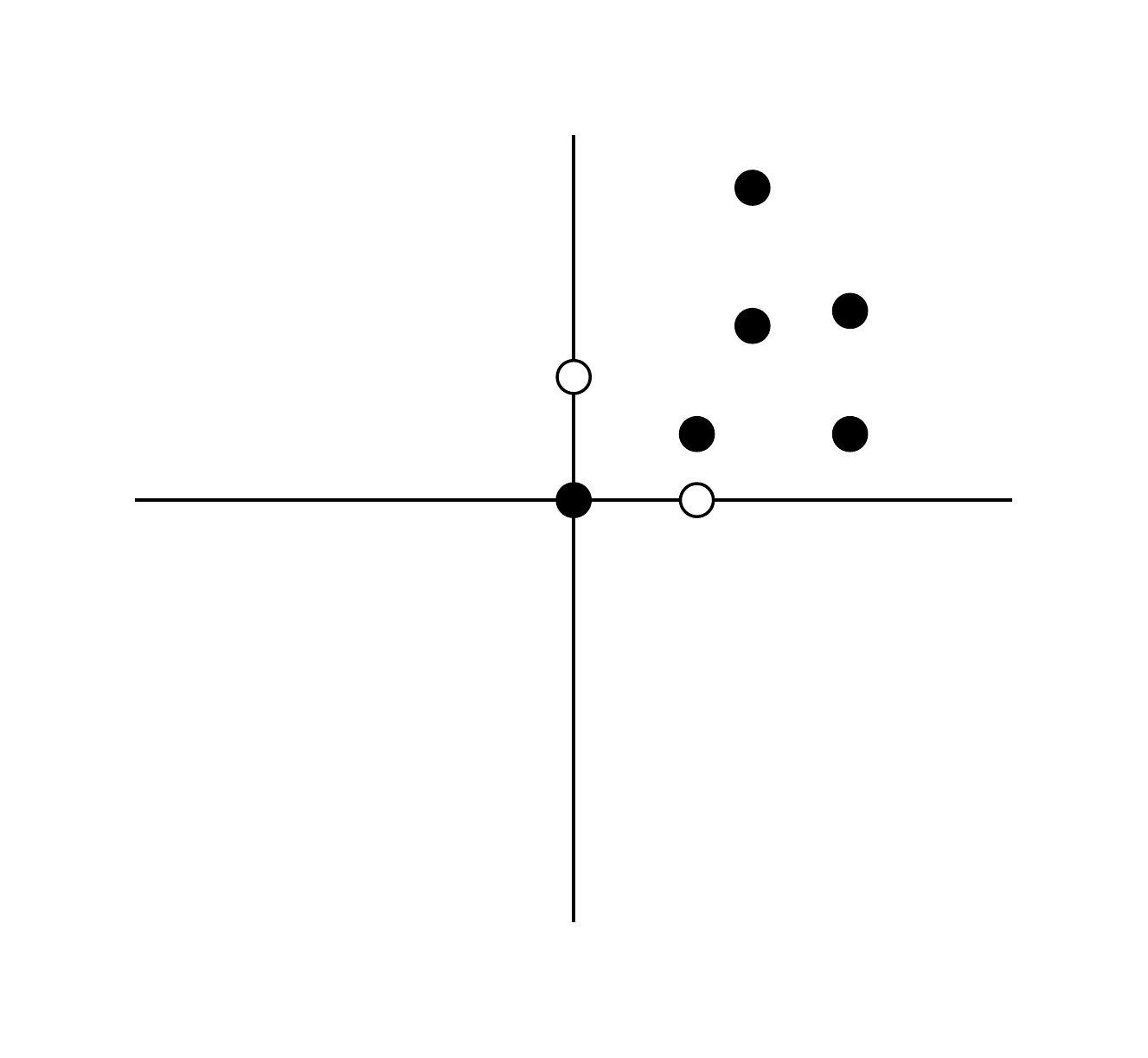}
        \caption{The situation described in Lemma~\ref{quadlem} where the white vertices are not resolved by any of the black vertices.}
\end{figure}

We can use this lemma to achieve the following result that is specific to $3$-minimals. 

\begin{lemma}\label{linelem}
A 3-minimal must have at least two vertices on the same line.
\end{lemma}

\begin{proof}
Suppose we have a set $M = \{(x_1,y_1),(x_2,y_2),(x_3,y_3)\}$ where no two vertices are on the same line, i.e.m $x_1 \neq x_2 \neq x_3$ and $y_1 \neq y_2 \neq y_3$. Without loss of generality, we let $x_1 < x_2 < x_3$. Now we have $y_i < y_j < y_k$, where $i,j,k \in \{1,2,3\}$ and $i\neq j \neq k$.\\
If we let $j=1$, then we have either $y_3<y_1<y_2$ or $y_2<y_1<y_3$. In either case, by Lemma~\ref{quadlem}, there are two neighbours of $(x_3,y_3)$ that are not resolved by any vertex in $M$. Therefore, $M$ is not a resolving set.\\
If we let $j=2$, then we have either $y_1<y_2<y_3$ or $y_3<y_2<y_1$. In either case, by Lemma~\ref{quadlem}, there are two neighbours of $(x_1,y_1)$ and two neighbours of $(x_3,y_3)$ that are not resolved by any vertex in $M$. Therefore, $M$ is not a resolving set.\\
Finally, if we let $j=3$, then we have either $y_1<y_3<y_2$ or $y_2<y_3<y_1$. In either case, by Lemma~\ref{quadlem}, there are two neighbours of $(x_1,y_1)$ that are not resolved by any vertex in $M$. Therefore, $M$ is not a resolving set if it contains any three vertices that are not on the same line. 
\end{proof}

We can now get the final property of $3$-minimals.

\begin{proposition}\label{lineprop}
A 3-minimals has either:
\begin{enumerate}[(i)]
\item Two vertices on the same line, $(i,j)$ and $(k,j)$, where  $i \neq k$, and a third vertex $(p,q)$, where $i\leq p \leq k$ and $q \neq j$. 
\item Two vertices on the same line, $(i,j)$ and $(i,k)$, where $j \neq k$, and a third vertex $(p,q)$, where $j\leq p \leq k$ and $p \neq i$.
\end{enumerate} 
\end{proposition}

\begin{proof}
Suppose we have a 3-minimal $M$. We know from Lemma~\ref{linelem} that two vertices in $M$ must be on the same line. Without loss of generality, we can let these vertices be $(i,j)$ and $(k,j)$ on a horizontal line since horizontal lines are equivalent to vertical lines in a rotated grid. We can also assume that $i<j$. Let the third vertex in $M$ be $(p,q)$.\\
If $q = j$ then $p \neq i,j$ since we must have 3 distinct vertices in a 3-minimal. However, this implies that there are two vertical neighbours $(p,q)$, denoted $(p,q^+)$ and $(p,q^-)$, that are not on the same line as $(i,j)$ and $(k,j)$. Therefore, by Corollary~\ref{c_1}, $(p,q^+)$ and $(p,q^-)$ are not resolved by any vertex in $M$ which is a contradiction.\\
Now suppose $q \neq j$ and $p<i$. This implies that $p<k$. Hence by Lemma~\ref{quadlem}, since $i,k > p$ and either $q<j$ or $q>j$, $M$ is not a resolving set. A similar argument holds for $p>k$.\\
Therefore $i\leq p \leq k$. 
\end{proof}

We now have enough results to give a complete characterisation of the 3-minimals:

\begin{theorem}
A set $M$ is a 3-minimal if and only if:
\begin{enumerate}[(i)]
\item $M$ has no more than one corner vertex.
\item $M$ contains two boundary vertices on opposite sides.
\item $M$ either has two vertices on the same line, $(i,j)$ and $(k,j)$ where  $i \neq k$, and a third vertex $(p,q)$, where $i\leq p \leq k$ and $q \neq j$, \\
or two vertices on the same line, $(i,j)$ and $(i,k)$, where $(j \neq k)$, and a third vertex $(p,q)$, where $j\leq p \leq k$ and $p \neq i$.
\end{enumerate} 
\end{theorem} 
\begin{proof}
It has already been shown from Propositions  \ref{cornerprop}, \ref{opprop}, and \ref{lineprop} that if any of the above conditions are not satisfied, then $M$ is not a 3-minimal. Hence, in order to prove the above statement, we need only show that if $M$ satisfies all three of the above conditions, then $M$ is a 3-minimal.\\
Suppose we have a minimal $M$ that satisfies the above conditions. By condition (ii), $M$ contains two boundary vertices, $u$ and $v$ on opposite sides. There are two possible cases: either $u$ and $v$ are on the same line, or $u$ and $v$ are on different lines.\\
Suppose $u$ and $v$ are on the same line. Clearly, neither $u$ or $v$ can be a corner without breaking condition (i). The third vertex, $w$, can be any vertex that is not on the line between $u$ and $v$ to satisfy condition (iii). The line between $u$ and $v$ divides the grid into two subgrids, $A$ and $B$, where the line is a side in each subgrid and $u$ and $v$ are corners of the side. Since two corners is a metric basis for a grid, $u$ and $v$ will resolve every pair of vertices in $A$ and every pair of vertices in $B$. Let $a \in A$ and $b \in B$ be a pair of distinct vertices that are unresolved by $u$ and $v$. This implies that the pair $a$ and $b$ are equidistant from the line between $u$ and $v$ and that $a$ and $b$ are on a line that is perpendicular to the line between $u$ and $v$. However, since $w$ is not on the line between $u$ and $v$, it must lie exclusively in $A$ or exclusively in $B$. Hence, if $w \in A$ then $d(w,a) < d(w,b)$, and if $w \in B$ then $d(w,b) < d(w,a)$. Thus $a$ and $b$ are resolved by $w$, and therefore every pair of vertices in the grid are resolved by $u$,$v$ and $w$.\\
Now suppose $u$ and $v$ are not on the same line. This implies that the third vertex $w$ is on the same line as $u$ or $v$, so without loss of generality, we let $w$ be on the same line as $u$. The vertex $w$ will either be on the same side as $u$ or on the line perpendicular to the side containing $u$. If the latter situation were the case, then $w$ would need to be on the same side as $v$ since in order to satisfy condition (iii), we are required to have a vertex of the line segment between $u$ and $w$ to be on the same line as $v$. However, this would give the same situation as in the case described above where we have opposite boundary vertices on the same line, hence $u$,$v$ and $w$ would resolve the grid. We therefore assume $w$ is on the same side as $u$. Without loss of generality (since the grid can be rotated), let $u$ be labelled by $(0,i)$ and $w$ be labelled by $(0,j)$, where $i<j$. Hence $v$ will be labelled $(p,q)$, where $p = m-1$ or $n-1$, and $i < q < j$.\\ 
Consider the subgrid that has $u$ and $w$ as corners and has boundary vertices as the other two corners. Every pair of vertices in this subgrid will be resolved by $u$ and $w$ since they are two corners that share a side of the subgrid. We refer to this subgrid as the middle subgrid, and we refer to the subgrid of all vertices whose horizontal coordinates are less than or equal to $i$ as the left subgrid, and the subgrid whose horizontal coordinates are greater than or equal to $j$ as the right subgrid. Every pair of vertices in the left subgrid are resolved by $u$ and $v$ since $u$ is a corner of this subgrid, and for every vertex $l$ in the left subgrid, there is a shortest path from $v$ to $l$ that goes through the other corner of the left subgrid that is on the same vertical line as $u$ (hence $v$ resolves the same vertices in the left subgrid that this corner would). Similarly, every pair of vertices in the right subgrid are resolved by $w$ and $v$.\\
Now suppose we have a pair of vertices $a$ and $b$, where $a$ is in the left subgrid and $b$ is in the middle subgrid, and suppose $a$ and $b$ are not resolved by $u$. If $a$ and $b$ were on the same horizontal line, then they would be equidistant from the vertical line containing $u$ which is a boundary of the left and middle subgrids. Therefore, since $w$ is in the middle subgrid but not in the left subgrid, $w$ will resolve $a$ and $b$ (as would $v$). If $a$ and $b$ were on different horizontal lines, then either $a$ would have a larger vertical coordinate than $b$, or $b$ would have a larger vertical coordinate than $a$. If $a$ had the larger vertical coordinate then $a$ and $b$ would be resolved by $w$ since $b$ and $w$ are closer together than $a$ and $w$ in both the horizontal coordinate and vertical coordinate. If $b$ had the larger vertical coordinate, then $v$ would resolve $a$ and $b$ since $b$ and $v$ are closer together than $a$ and $v$ in both the horizontal coordinate and vertical coordinate. Hence $a$ and $b$ are always resolved. Similarly, if $a$ were in the right subgrid with $b$ still in the middle subgrid, then $a$ and $b$ would be resolved.\\ 
The last case is when $a$ is in the left subgrid and $b$ is in the right subgrid. Suppose $a$ and $b$ are equidistant from $u$. If $a$ and $b$ are on the same horizontal line, then $a$ and $b$ are resolved by $w$ since $b$ and $w$ will be closer together than $a$ and $w$. If $a$ and $b$ were on different horizontal lines with $a$ having a larger vertical coordinate than $b$, then as before, $a$ and $b$ would be resolved by $w$ since $b$ and $w$ are closer together than $a$ and $w$. And similarly, if $b$ had a larger vertical coordinate than $a$, then $a$ and $b$ would be resolved by $v$.\\
Hence any pair of vertices $a$ and $b$ are resolved by $u$,$v$ and $w$ if the three conditions hold.
\end{proof}

\begin{figure}[H]
        \centering
        \begin{subfigure}[b]{0.3\textwidth}
                \includegraphics[width=\textwidth]{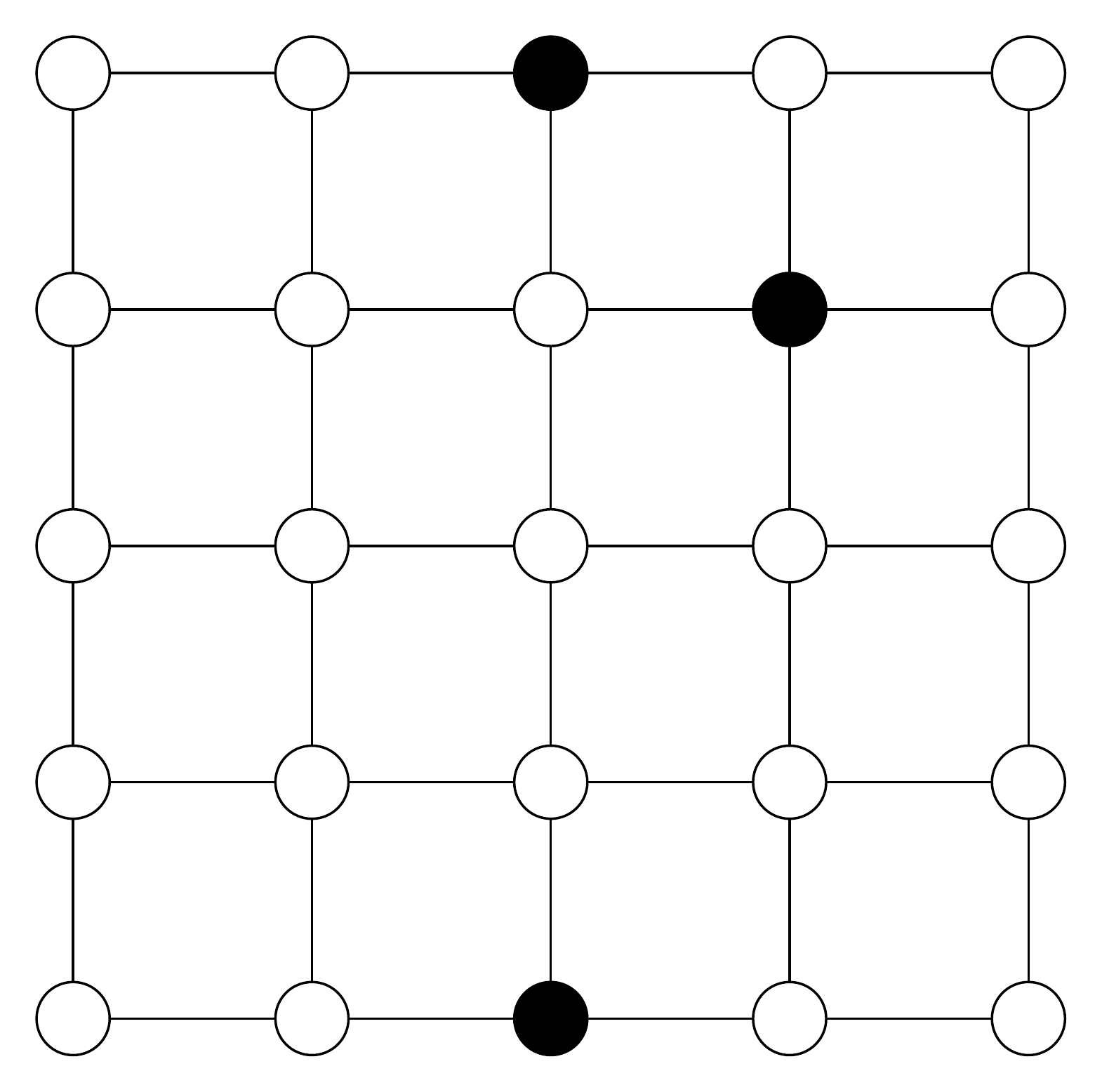}
        \end{subfigure}%
		\qquad
        \begin{subfigure}[b]{0.3\textwidth}
                \includegraphics[width=\textwidth]{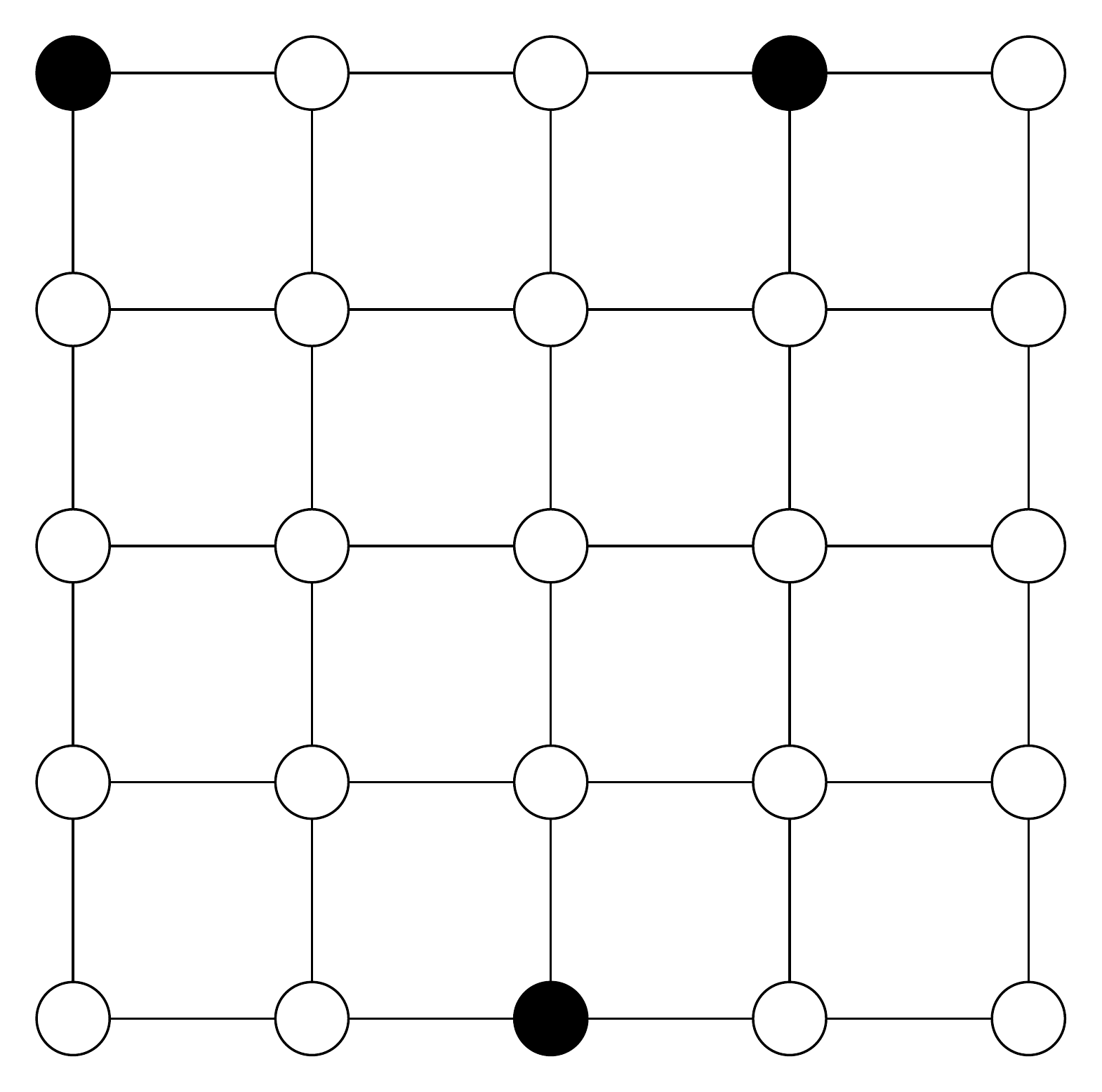}
        \end{subfigure}
        \caption{An example of the two types of 3-minimals: one with opposite boundary vertices on the same line, and one with opposite boundary vertices on different lines.}
\end{figure}

We now wish to find any $k$-minimals where $k>3$. In order to do this, we will need a more powerful version of Lemma~\ref{quadlem} which will also use a vertex as an origin and consider the quadrants with respect to this origin. First, we define the boundary of a quadrant to be the points on the two halves of the axes that define a quadrant, not including the origin, e.g., if the origin were $(0,0)$ then the points $(x,0)$ for $x>0$ and $(0,y)$ for $y>0$ would be on the boundary of the first quadrant. We say that two quadrants are opposite if they have no boundary points in common, e.g. the first and third quadrants are opposite, otherwise we say they are adjacent Also, the quadrant boundaries are not considered to be within any quadrant and neither is the origin. Now we have the following lemma:
\begin{lemma} \label{quadlem2}
Suppose we have an interior vertex $(p,q)$. Let $(p,q)$ be the origin of the coordinate axes $y=q$ and $x=p$. If the vertices $(x_1^+,y_1^+),(x_2^+,y_2^+),\dots ,(x_{k_1}^+,y_{k_1}^+)$  are in the same quadrant with respect to the origin $(p,q)$, if the vertices $(x_1^-,y_1^-),\\(x_2^-,y_2^-),\dots ,(x_{k_2}^-,y_{k_2}^-)$ are in the opposite quadrant, and if the vertices $(p,q^+_1),\\(p,q^+_2),\dots ,(p,q^+_{k_3}),(p^+_1,q),(p^+_2,q),\dots (p^+_{k_4},q)$ are the boundary points of one of these quadrants, then there exist two neighbours of $(p,q)$, denoted by $(p^*,q)$ and $(p,q^*)$, that are not resolved by any of the vertices $(p,q),(x_1^+,y_1^+),(x_2^+,y_2^+),\dots ,\\(x_{k_1}^+,y_{k_1}^+),(x_1^-,y_1^-),(x_2^-,y_2^-),\dots ,(x_{k_2}^-,y_{k_2}^-),(p,q^+_1),
(p,q^+_2),\dots ,(p,q^+_{k_3}),(p^+_1,q),\\(p^+_2,q),\dots (p^+_{k_4},q)$.
\end{lemma}

\begin{figure}[H]
        \centering
        \includegraphics[width=0.5\textwidth]{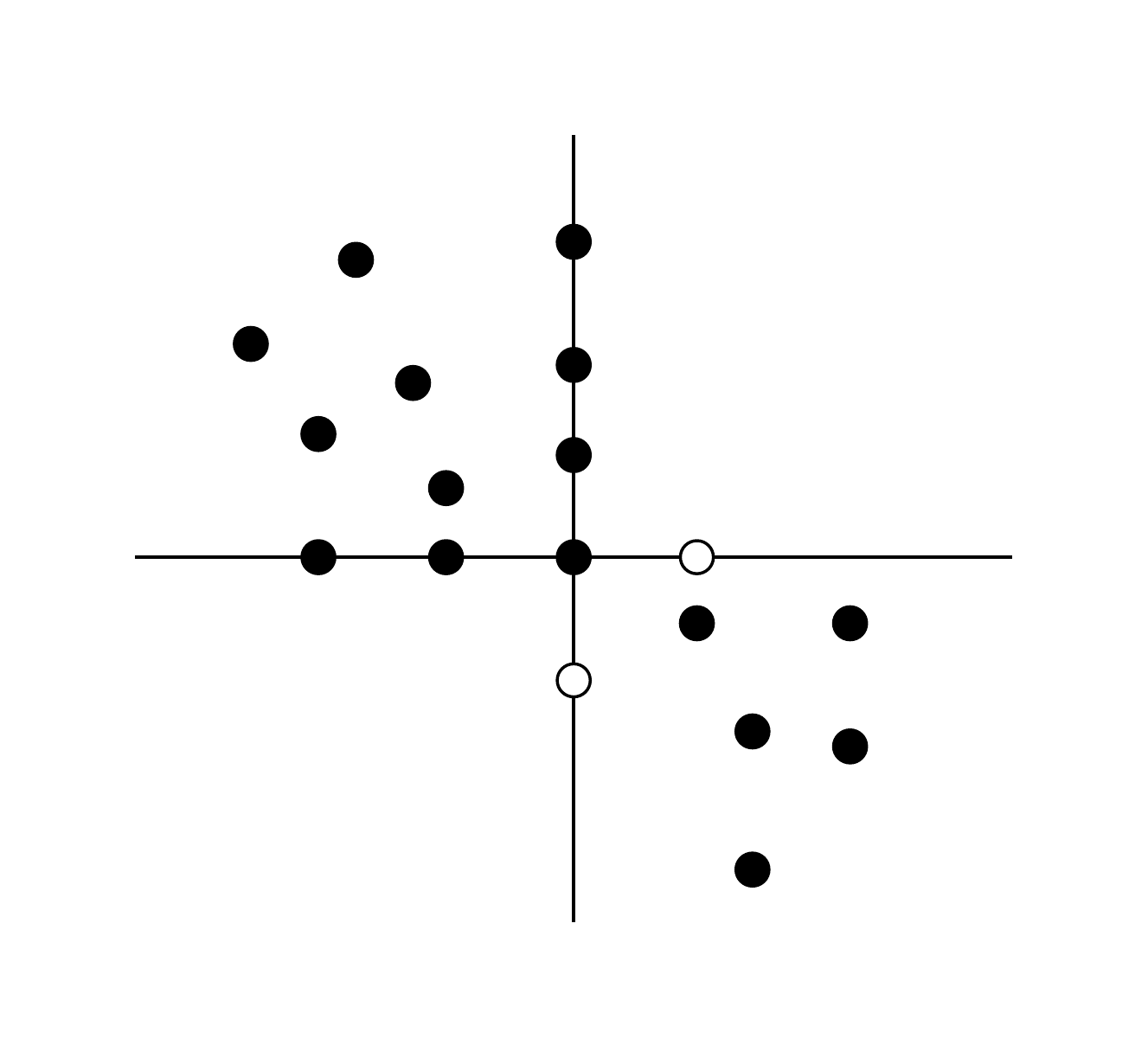}
        \caption{The situation described in Lemma~\ref{quadlem2}, where the white vertices are not resolved by any of the black vertices.}
\end{figure}

\begin{proof}
Assume without loss of generality that the vertices $(p,q^+_1), (p,q^+_2), \dots ,\\(p,q^+_{k_3}),(p^+_1,q), (p^+_2,q), \dots (p^+_{k_4},q)$ are the boundary points of the quadrant containing $(x_1^+,y_1^+),(x_2^+,y_2^+),\dots ,(x_{k_1}^+,y_{k_1}^+)$.\\
By Lemma~\ref{quadlem}, since the vertices $(x_1^-,y_1^-),(x_2^-,y_2^-),\dots ,(x_{k_2}^-,y_{k_2}^-)$ are in the same quadrant with respect to $(p,q)$, then there exist two vertices, which we will denote by $(p^-,q)$ and $(p,q^-)$ that are not resolved by any of the  vertices $(p,q),(x_1^-,y_1^-),(x_2^-,y_2^-),\dots ,(x_{k_2}^-,y_{k_2}^-)$.\\
Consider the shared neighbour of $(p^-,q)$ and $(p,q^-)$ which resides in the quadrant containing $(x_1^-,y_1^-),(x_2^-,y_2^-),\dots ,(x_{k_2}^-,y_{k_2}^-)$. We denote this vertex by $(p^-,q^-)$. Now the vertices $(x_1^+,y_1^+),(x_2^+,y_2^+),\dots ,(x_{k_1}^+,y_{k_1}^+),(p,q^+_1),(p,q^+_2),\dots ,\\(p,q^+_{k_3}),
(p^+_1,q),(p^+_2,q),\dots (p^+_{k_4},q)$ are all in the same quadrant with respect to $(p^-,q^-)$ as the origin, so by Lemma~\ref{quadlem}, there are two neighbours of $(p^-,q^-)$ that are unresolved by any of the vertices $(p^-,q^-),(x_1^+,y_1^+),(x_2^+,y_2^+),\dots ,(x_{k_1}^+,y_{k_1}^+),\\
(p,q^+_1),(p,q^+_2),\dots ,(p,q^+_{k_3}),(p^+_1,q),(p^+_2,q),\dots (p^+_{k_4},q)$. These neighbours must be $(p^-,q)$ and $(p,q^-)$ since they lie on the boundary of the quadrant containing $(p^-,q^-)$, hence the vertices $(p^-,q) = (p^*,q)$ and $(p,q^-) = (p,q^*)$ are not resolved by any of the vertices $(p,q),(x_1^+,y_1^+),(x_2^+,y_2^+),\dots ,(x_{k_1}^+,y_{k_1}^+),
(x_1^-,y_1^-),\\
(x_2^-,y_2^-),\dots ,(x_{k_2}^-,y_{k_2}^-),
(p,q^+_1),(p,q^+_2),\dots ,(p,q^+_{k_3}),(p^+_1,q),(p^+_2,q),\dots (p^+_{k_4},q)$. 
\end{proof}

Lemmas~\ref{quadlem} and~\ref{quadlem2} can be used to show that a set of vertices does not resolve the grid by finding a vertex that has a pair of neighbours that are unresolved. If a vertex does not have a pair of neighbours that are unresolved, then we say the vertex has a \textit{locally resolved neighbourhood}. For general graphs, if every vertex has a locally resolved neighbourhood, we cannot say that the graph is resolved. However, it turns out that for grid graphs, we are allowed to make this conclusion.

\begin{theorem}\label{localthm}
If  $G=(V,E)$ is a grid and $R \subseteq V$ is a set of vertices such that every vertex in $G$ has a locally resolved neighbourhood with respect to $R$, then $R$ is a resolving set for $G$. 
\end{theorem}

\begin{proof}
The proof of this theorem is by induction. We start with the graph $G = P_3 \square P_3$ and attempt to construct a set $R$ that gives every vertex in the grid a locally resolved neighbourhood. The proof of Proposition~\ref{opprop} shows that the corners of a grid do not have locally resolved neighbourhoods if we do not have two boundary vertices on opposite sides as landmarks. Therefore $R$ must contain two boundary vertices on opposite sides. If these vertices are two corners on the same side, then $G$ would be resolved and so we are done. The proof of Proposition~\ref{cornerprop} shows that these two vertices will not locally resolve the grid if they are corners on opposite sides, so without loss of generality, since reflections and rotations do not change the grid, we have two cases, as shown in Fig~\ref{localfig}. 

\begin{figure}[H]
        \centering
        \begin{subfigure}[b]{0.45\textwidth}
                \includegraphics[width=\textwidth]{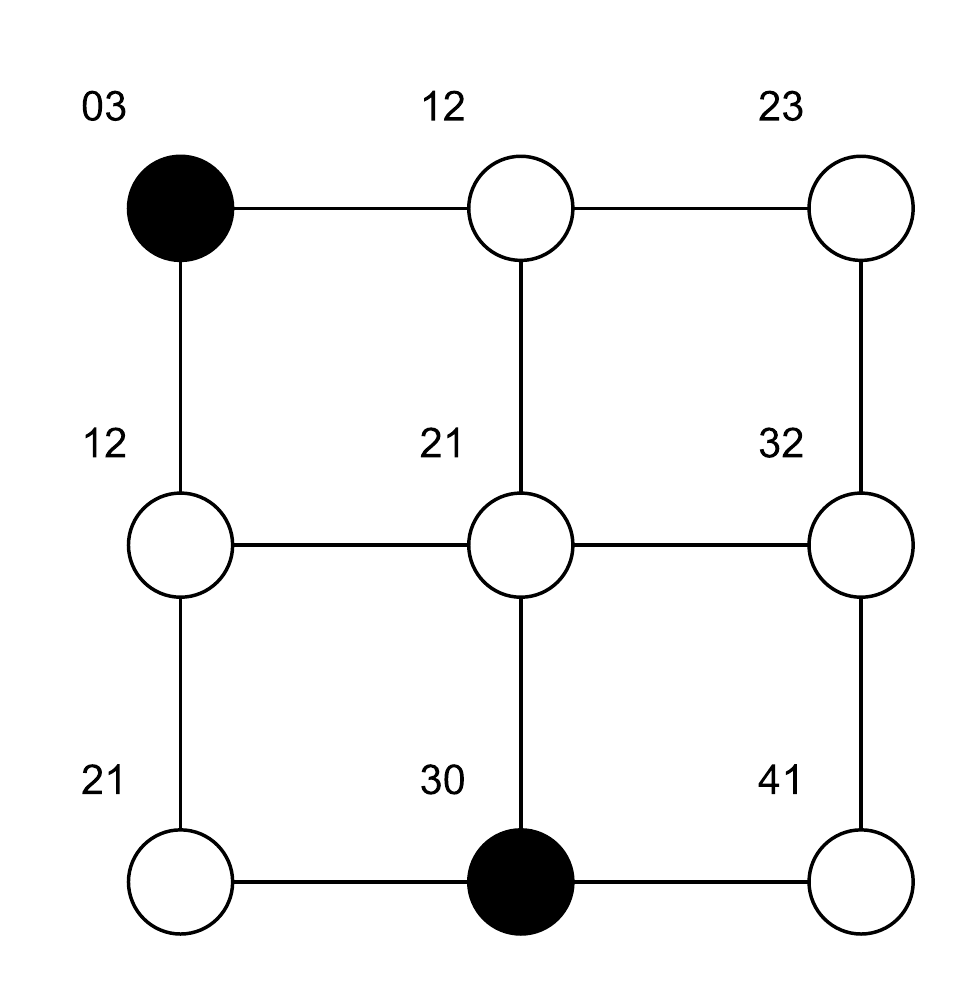}
        \end{subfigure}%
		\qquad
        \begin{subfigure}[b]{0.45\textwidth}
                \includegraphics[width=\textwidth]{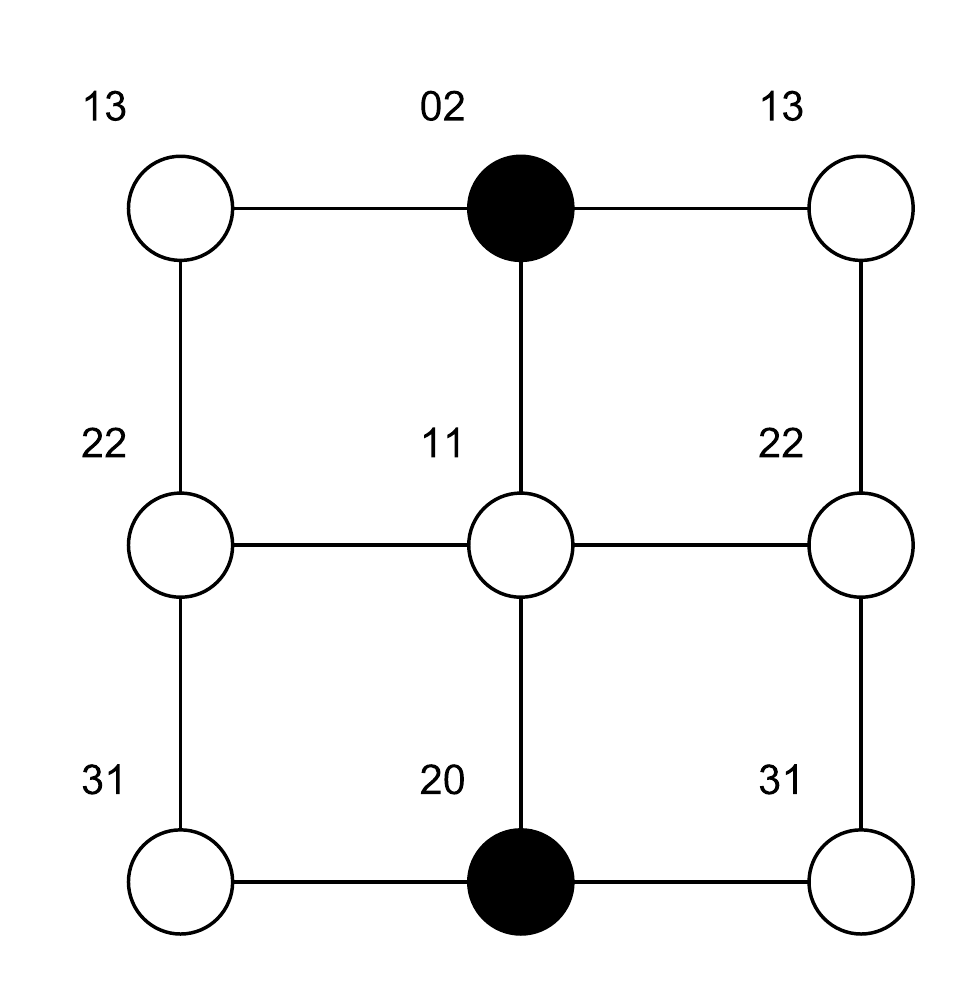}
        \end{subfigure}
        \caption{The black vertices are the current elements of $R$. The vertex labels are the shortest distances from the black vertices.}
        \label{localfig}
\end{figure}
In both of these cases, the pairs of vertices that are unresolved are pairwise disjoint and for each of these pairs, there is a vertex that has both members of the pair as neighbours. Hence if we added vertices to $R$ to give every vertex in $G$ a locally resolved neighbour, then $G$ would be resolved, hence the theorem is true for $G = P_3 \square P_3$.\\
Now we suppose that the theorem is true for some $G=A$ where $A$ is a grid graph. We extend this grid by adding an extra row/column of vertices which we will denote by the set $B$. Without loss of generality, we let $B$ be a new row placed at the bottom of $A$. We denote this extended graph by $G^+$. Let $R$ be a set of landmark vertices in $G^+$ that gives every vertex in $G^+$ a locally resolved neighbourhood. Any pair of vertices in $A$ will be resolved by the induction hypothesis. Note that this remains true even if $R$ contained vertices in $B$ since a having a landmark $b \in B$ would be equivalent to having the vertex above $b$ as a landmark in $A$ when considering the resolvability of the $A$ subgrid. Suppose we have a pair of vertices in $B$ that is not resolved by any landmark. Let this pair be $(b_1,b_2)$. If $(b_1,b_2)$ is not resolved by any vertex in $B$ then this implies that there is a pair of vertices in $A$, denoted by $(a_1,a_2)$, which is not resolved by any landmark in $B$ where $a_1$ is the vertex above $b_1$ and $a_2$ is the vertex above $b_2$. This is because there is always a shortest path from a landmark in $B$ to $a_1$ that passes through $b_1$ (and similarly for $a_2$ and $b_2$). However, $(a_1,a_2)$ would also not be resolved by any landmark in $A$ since for every landmark in $A$, a shortest path to $b_1$ will have $a_1$ as the second last vertex (and similarly for $b_2$ and $a_2$). Hence if $(b_1,b_2)$ is not resolved by any landmark in $G^+$ then $(a_1,a_2)$ is not resolved which contradicts the induction hypothesis. Thus every pair of vertices in $B$ is resolved.\\
Now we need only consider the pairs of vertices $(a,b)$ where $a \in A$ and $b \in B$. Let $b$ be a vertex that is between two landmarks in $B$. Let $b^-$ be the landmark to the left of $b$ and let $b^+$ be the landmark to the right of $b$. Suppose the pair $(a,b)$ was not resolved by $b^+$. Since $a$ is at least one row above $b^+$ it must be at least one column to the right of $b$ since $b$ is to the left of $b^+$ on the same horizontal line and $d(b^+,b)=d(b^+,a)$. However, this implies that there is a shortest path from $b^-$ to $a$ that goes through $b$ since $b^-$ and $b$ are on the same horizontal line, $b$ is to the right of $b^-$, and $a$ is to the right of $b$. Hence $d(b^-,b) \neq d(b^-,a)$ so $(a,b)$ is resolved by $b^-$. Equivalently, if $(a,b)$ were not resolved by $b^-$ then the pair would be resolved by $b^+$.\\
Now suppose that $b$ is not between any two or more landmarks in $B$. This means that $b$ is to the left of the leftmost landmark in $B$, to the right of the rightmost vertex in $B$, or $B$ contains no landmarks. If $B$ contains a landmark, then without loss of generality, let $b$ be to the left of the leftmost landmark in $B$. We will denote this landmark by $b^*$. If $B$ does not contain any landmarks then we let $b$ be any vertex in $B$ and, without loss of generality, we let $b^*$ be the right neighbour of $b$. The vertex $b^*$ has a locally resolved neighbourhood; it follows that the pair of vertices consisting of the neighbour to the left of $b^*$ and the neighbour above $b^*$ must be resolved. There are no landmarks in $B$ that will resolve this pair since the only possible landmarks in $B$ are $b^*$ and vertices to the right of $b^*$, which implies that there always exist shortest paths from any landmark in $B$ to each these two neighbours of $b^*$ that contain $b^*$ as the second last vertex in the path. Furthermore, Lemma~\ref{quadlem} implies that no vertex in $A$ that is to the left of $b^*$ will resolve this pair. Hence there must be a landmark in $A$, which we will denote by $a^*$, that is either directly above $b^*$ on the same vertical line or to the right of $b^*$. Suppose $(a,b)$ is not resolved by $b^*$. This means that $d(b^*,b)=d(b^*,a)$. Since $a^*$ is directly above or to the right of $b^*$ and $b^*$ is to the right of $b$ on the same horizontal line, $d(a^*,b)=d(a^*,b^*) + d(b^*,b) = d(a^*,b^*) + d(b^*,a)$. If $(a,b)$ was not resolved by $a^*$, then $d(a^*,a) = d(a^*,b) \Rightarrow d(a^*,a)=d(a^*,b^*) + d(b^*,a)$ which is a contradiction since $b^*$ is below both $a$ and $a^*$ so no shortest path from $a^*$ to $a$ would contain $b^*$. Hence $(a,b)$ is resolved by $a^*$ and thus any pair $(a,b)$ is resolved by landmarks in $R$.\\
Therefore $G^+$ is resolved by the landmarks in $R$.
\end{proof}

If we let a vertex in the grid be the origin of a set of axes and consider the landmarks with respect to this origin as we did in Lemmas~\ref{quadlem} and~\ref{quadlem2}, then there are only a few different types of situations where the vertex is locally resolved and so we can use this to give a weak characterisation of an arbitrary resolving set of a grid graph. Satisfying Proposition~\ref{opprop} will guarantee that the corner vertices have locally resolved neighbourhoods, so we need only consider the situations for side and interior vertices.\\
For side vertices, we consider the situations that differ from the one described in Lemma~\ref{quadlem} since we know that this situation will not locally resolve a side vertex. This leaves us with two cases:
\begin{itemize}
\item \textbf{Side Case (1)}: There are landmarks in two different quadrants.
\end{itemize}
The other situation that differs from the one described in Lemma~\ref{quadlem} is when we have a landmark on a quadrant boundary. Denote this quadrant boundary by $q$. This alone will not give a locally resolved neighbourhood, so we must include an additional landmark vertex somewhere other than $q$. We cannot put the additional landmark in a quadrant that does not have $q$ as a boundary as this will leave the same pair that was unresolved by the first landmark unresolved. Thus we get the following case:
\begin{itemize}
\item \textbf{Side Case (2)}: There is a landmark on a quadrant boundary $q$ and an additional landmark that is not in $q$ or in the quadrant that does not have $q$ as a boundary.
\end{itemize}

\begin{figure}[H]
        \centering
        \begin{subfigure}[b]{0.4\textwidth}
                \includegraphics[width=\textwidth]{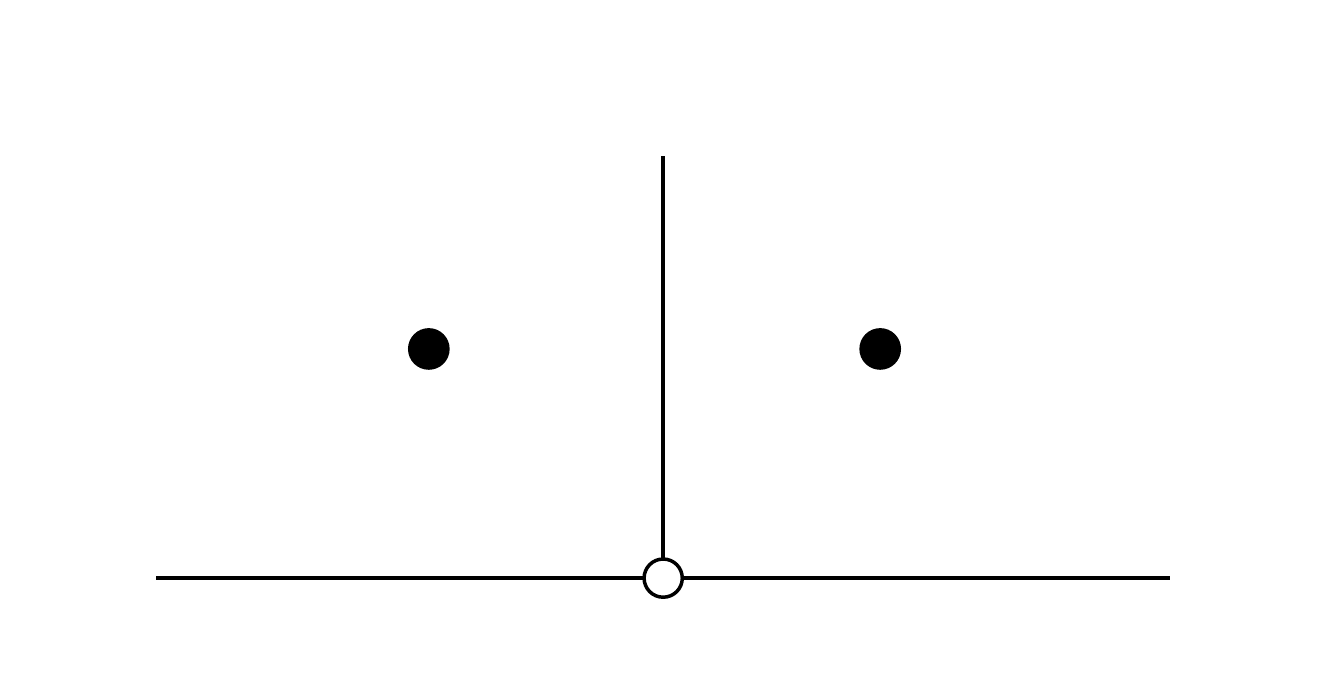}
                \caption{Side Case (1)}
        \end{subfigure}%
		\qquad
        \begin{subfigure}[b]{0.4\textwidth}
                \includegraphics[width=\textwidth]{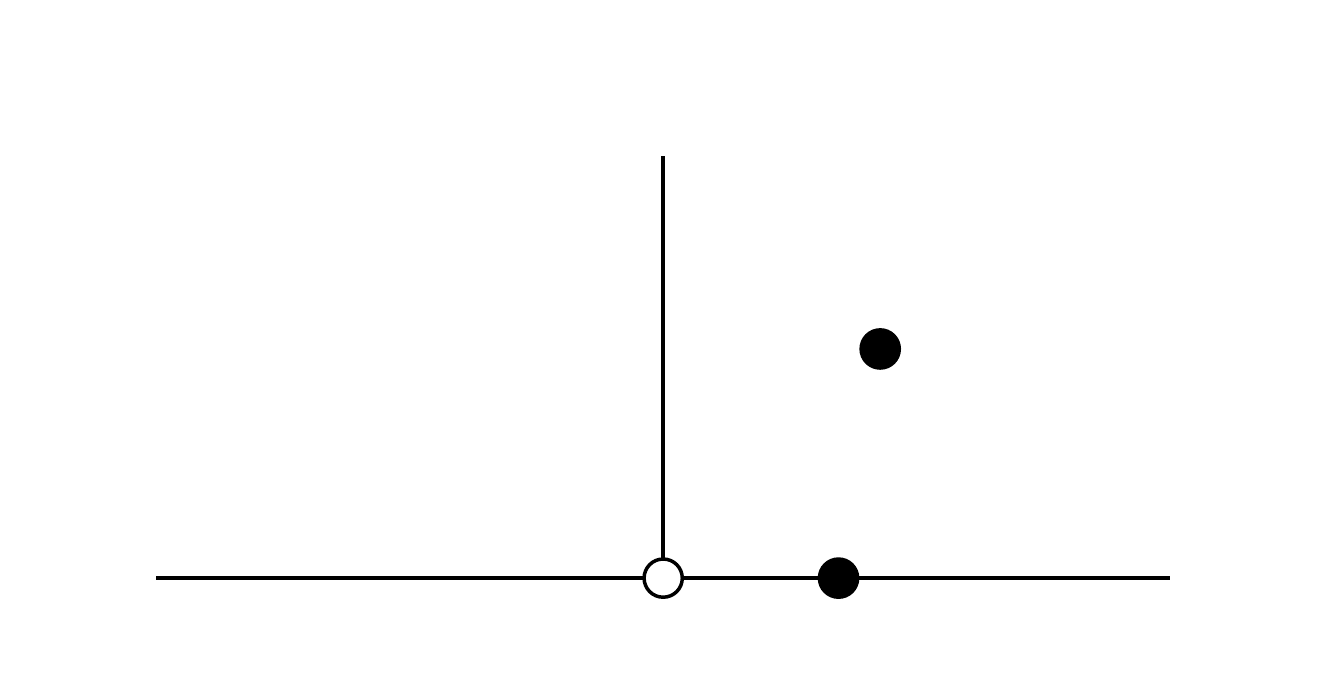}
                \caption{Side Case (2)}
        \end{subfigure}
\end{figure}

For interior vertices, we consider the situations that differ from the one described in Lemma~\ref{quadlem2}. This leaves us with three cases:
\begin{itemize}
\item \textbf{Interior Case (1)}: There are landmarks in two adjacent quadrants.
\item \textbf{Interior Case (2)}: There are a landmarks on both boundaries of the same quadrant and another landmark that is in an adjacent quadrant.
\end{itemize}
The only other situation that differs from the one described in Lemma~\ref{quadlem2} is when we have landmarks in two different quadrant boundaries that do not share a quadrant. Denote these quadrant boundaries by $p$ and $q$. This alone will not give a locally resolved neighbourhood but putting an additional landmark anywhere except $p$ and $q$ will. Therefore, we get the following case.
\begin{itemize}
\item \textbf{Interior Case (3)}: There are landmarks in two different quadrant boundaries $p$ and $q$ that do not share a quadrant and an additional landmark that is not in $p$ or $q$.
\end{itemize}

\begin{figure}[H]
        \centering
        \begin{subfigure}[b]{0.4\textwidth}
                \includegraphics[width=\textwidth]{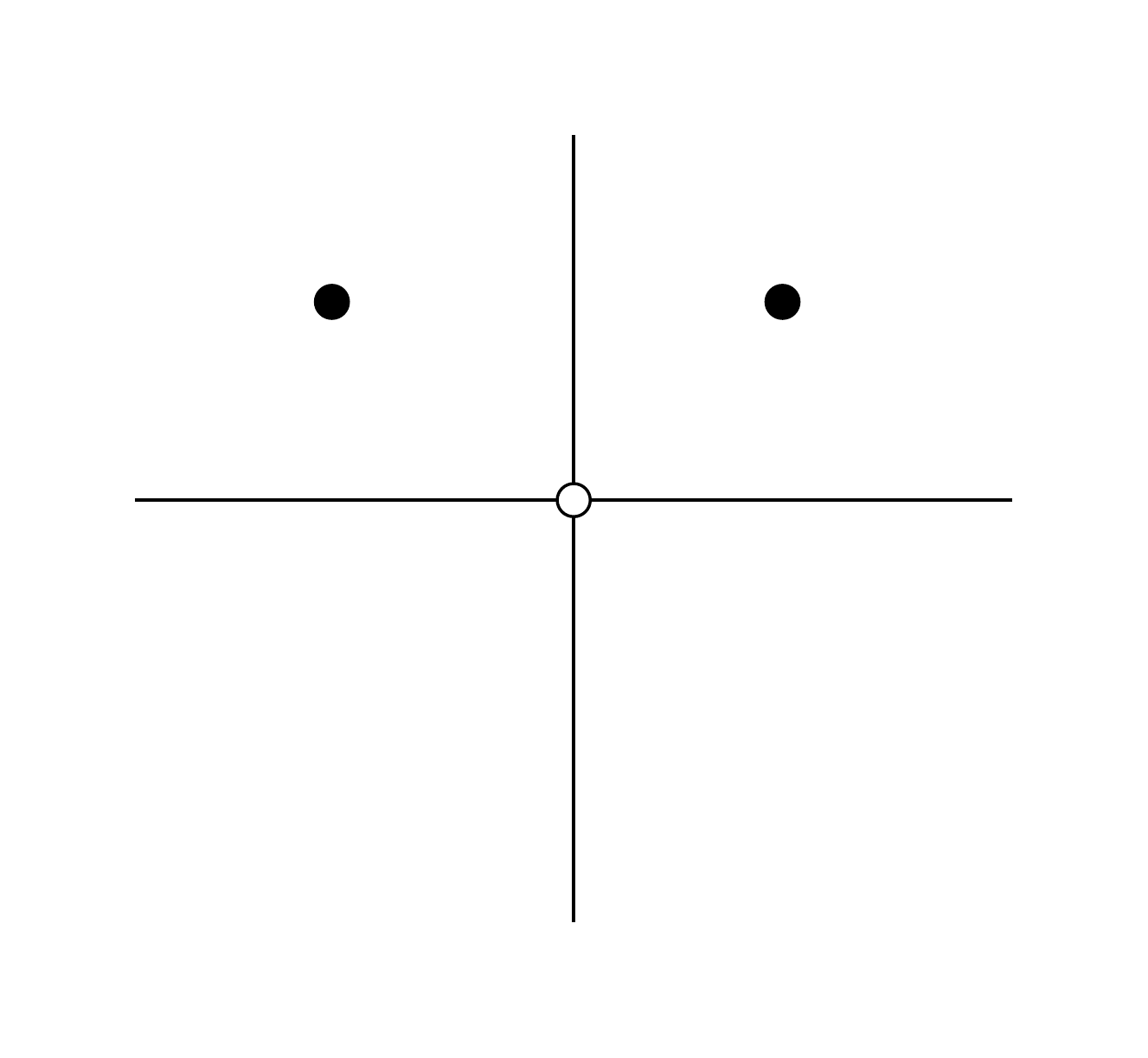}
                \caption{Interior Case (1)}
        \end{subfigure}%
		\qquad
        \begin{subfigure}[b]{0.4\textwidth}
                \includegraphics[width=\textwidth]{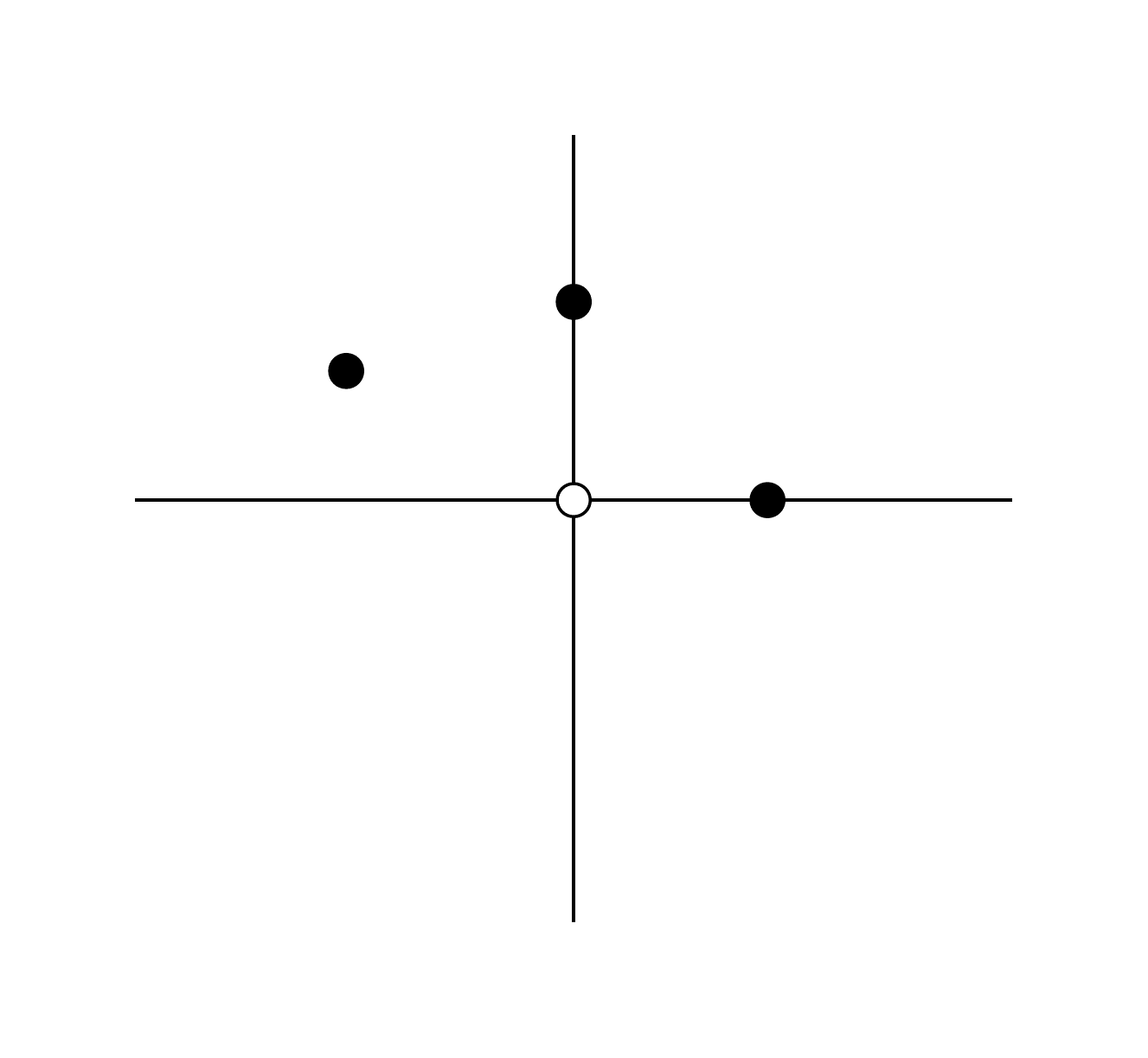}
                \caption{Interior Case (2)}
        \end{subfigure}
        \qquad
        \begin{subfigure}[b]{0.4\textwidth}
                \includegraphics[width=\textwidth]{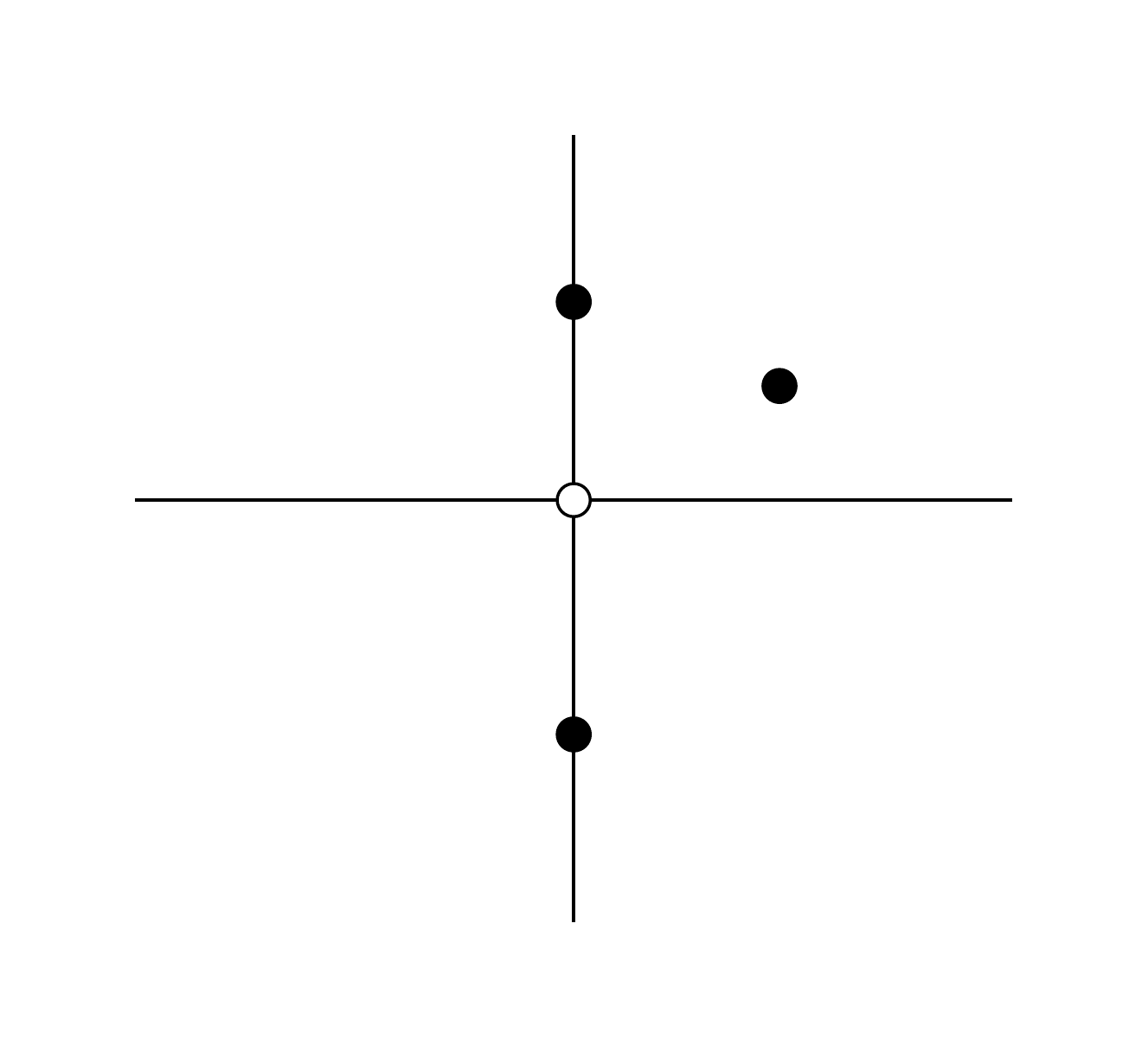}
                \caption{Interior Case (3)}
        \end{subfigure}
\end{figure}

It can easily be verified that Side Cases (1) and (2) and Interior Cases (1), (2) and (3) give the origin a locally resolved neighbourhood. Hence our weak characterisation of an arbitrary resolving set of a grid graph is a set of vertices that contains two boundary vertices on opposite sides, and in which every side vertex is in a situation described by Side Case (1) or (2) and every interior vertex is in a situation described by Interior Case (1), (2) or (3) with respect to the vertices in the set.\\
It is possible to use our weak characterisation to find classes of $k$-minimals. We have characterised one such class. In order to describe the characterisation of this class of $k$-minimals, we first need some additional definitions.\\
A \textit{line segment} between two vertices $a$ and $b$ that are on the same line is the unique shortest path between $a$ and $b$. A line segment can either be a horizontal line segment or a vertical line. Now suppose we have a set of vertices $R$ on the grid. We define a \textit{horizontal line segment path} between two vertices $u$ and $v$ with respect to $R$, where $u,v \in R$, to be a shortest path from $u$ to $v$ that only uses horizontal line segments between the vertices in $R$ and the vertical lines that intersect these line segments (if such a path exists). If there are no vertices in $R$ on the same horizontal line as a vertex $w \in R$, then $w$ is a horizontal line segment of length one. A similar definition exists for the vertical line segment path that instead uses vertical line segments and the horizontal lines that intersect them. We say that the horizontal (vertical) line segment path is minimal if the horizontal (vertical) line segment path between $u$ and $v$ with respect to $R$ exists, but no horizontal (vertical) line segment path exists between $u$ and $v$ with respect to any set $R-\{w\}$, where $w \in R$.\\
Let $X_1,X_2,\dots,X_l$ be the sets of horizontal coordinates of the vertices in each of the horizontal line segments between vertices in then set $R$. If the horizontal coordinate of $u$ is $p$ and the horizontal coordinate of $v$ is $q$, where $u,v \in R$ and $p<q$, then there is no horizontal line segment path from $u$ to $v$ with respect to $R$ if $\{p,p+1, \dots,q-1, q\}\nsubseteq \bigcup_{i=1}^l X_i$. Below we have given some necessary conditions that must be satisfied in order for this horizontal line segment path to be minimal:

\begin{enumerate}[(1)]
\item There are no more than two vertices in the same row since we would only need the pair of vertices with the largest horizontal distance between them.
\item $X_i \cap X_j \cap X_k = \emptyset$ for any three horizontal line segments, otherwise we could achieve the same result using only two of the horizontal line segments.
\item  $X_i \nsubseteq X_j$ for any two horizontal line segments, otherwise we could remove the line segment with the $X_i$ as the horizontal coordinates and we would still achieve the same result.
\item The largest horizontal coordinate in any $X_i$ is greater or equal to $p$.
\item The smallest horizontal coordinate in any $X_i$ is less or equal than $q$.
\item If $u$ is above $v$, then if horizontal line segment with horizontal coordinates $X_i$ is above the horizontal line segment with coordinates $X_j$, where neither of these horizontal line segments contain $u$ or $v$, then the largest horizontal coordinate in $X_j$ must be strictly greater than the largest horizontal coordinate in $X_i$. If we had this situation and the line segment with coordinates $X_i$ were necessary for the horizontal line segment path, then there would exist such a path that would not use the line segment with coordinates $X_j$.
\end{enumerate} 
  
We now use the above definitions to give the characterisation of a class of $k$-minimals:

\begin{theorem} \label{segementthm}
A set M of cardinality $k>3$ is a $k$-minimal if:
\begin{enumerate}[(i)]
\item $M$ has no more than one corner vertex.
\item $M$ contains two boundary vertices, $u$ and $v$, on opposite sides that are not on the same line. Furthermore, $M$ does not contain any other pair of vertices on opposite sides.
\item There is a minimal horizontal line segment path between $u$ and $v$ with respect to $M$ if $u$ and $v$ are on horizontal sides, otherwise there is a minimal vertical line segment path between $u$ and $v$ with respect to $M$. 
\end{enumerate}
\end{theorem}

\begin{proof}
Let $M$ be a set of vertices that satisfies the above conditions. We will assume without loss of generality that the vertices $u$ and $v$ are on horizontal sides of the grid, the side containing $u$ is above the side containing $v$, and $u$ is to the left of $v$. Hence there is a horizontal line segment path between $u$ and $v$ with respect to $M$. We will refer to such paths as $M$-paths. We will prove the resolvability of $M$ by showing that every vertex has a locally resolved neighbourhood with respect to the elements of $M$. Since condition (ii) gives the corners locally resolved neighbourhoods, we need only consider the side and interior vertices.\\
If a side vertex $w$ is not on the same side as $u$ or $v$, then $w$ is in the situation described by Side Case (1) where $u$ and $v$ are in different quadrants. If the side vertex $w$ is on the same line as $u$ (but not $u$), then $u$ is in a quadrant boundary with respect to $w$ as the origin and there is either a vertex below $u$ on the same vertical line, or there are vertices below and at either side of $u$ in order for $u$ to be above a horizontal line segment. In either case, $w$ is in the situation described by Side Case (2)(and similarly if the side vertex is on the same side as $v$). Finally, if the side vertex $w$ is either $u$ or $v$ then it is either in the situation described by Side Case (2) if there is a vertex on the same vertical line as $w$, or it is in Side Case (1) if $w$ is above a horizontal line segment. Hence all the side vertices have locally resolved neighbourhoods.\\
Now consider the interior vertices. If $w$ is an interior vertex that is not in an $M$-path, then it is in the situation described by Interior Case (1) since the horizontal line segment path from $u$ to $v$ would cross the horizontal axis with $w$ as the origin. If $w$ is an interior vertex in an $M$-path but $w \notin M$, then there are 3 cases:
\begin{enumerate}[(1)]
\item $w$ is between two vertices $m_1,m_2 \in M$ on the same horizontal line.
\item $w$ is between a vertex in $m \in M$ and $c \notin M$ that is in a horizontal line segment, where $c$ and $m$ are on the same vertical line.
\item $w$ is between vertices $c_1,c_2 \notin M$ that are in two different horizontal line segments, where $c_1$ and $c_2$ are on the same vertical line.
\end{enumerate}
In Cases (2) and (3), $w$ is in the situation of Interior Case (1) due to $w$ being on a vertical line that intersects a horizontal line segment. In case (1), $w$ is in the situation described by Interior Case (3).\\ 
Finally, if $w$ is an interior vertex in an $M$-path and $w \in M$, then there must be another vertex in $M$ on the same horizontal line as $w$. If the vertical line containing $w$ contains another vertex in $M$, then $w$ is in the situation described by Interior Case (2). Otherwise, $w$ is in the situation described by Interior Case (1). Hence, all interior vertices have locally resolved neighbourhoods with respect to $M$. Therefore, all the vertices in the grid have locally resolved neighbourhoods with respect to the vertices of $M$ and so by Theorem~\ref{localthm}, $M$ is a resolving set.\\
Now we will show that $M$ is a minimal. Suppose we removed a vertex $w$ from $M$. If $w$ was either of the vertices $u$ or $v$, then condition (ii) implies that $M-\{w\}$ would not contain a pair of boundary vertices on opposite sides of the grid so by Proposition~\ref{opprop}, $M-\{w\}$ is not resolving. Now assume that $w$ is neither $u$ nor $v$. We know that there is an $m \in M$ on the same line as $w$. We also know that no vertical line between $w$ and $m$ intersects with two or more other horizontal line segments and that one of the vertical lines intersects with a horizontal line segment that is above $m$ and $w$ and another of these vertical lines intersects with a horizontal line segment that is below $m$ and $w$. Furthermore, we know that one of the vertical lines between $m$ and $w$ that intersects with the line segment above $m$ and $w$ will intersect with the line segment at a vertex in $M$ due to conditions (3) and (6) of a minimal horizontal line segment path. Similarly, we know that one of the vertical lines between $m$ and $w$ that intersects with the line segment below $m$ and $w$ will intersect with the line segment at a vertex in $M$. If $w$ is to the left of $m$, let $w_1$ be the first vertex on horizontal line between $w$ and $m$ for which there is vertex in $M$ that is on the same vertical line and below $w_1$. If $w$ is to the right of $m$, let $w_1$ be the first vertex on horizontal line between $m$ and $w$ for which there is vertex in $M$ that is on the same vertical line and above $w_1$. In either case, if $w_1$ is an interior vertex then it is in the situation described by Lemma~\ref{quadlem2} as there is not a pair of vertices in $M$ on the same vertical line as $w_1$ that are above and below $w_1$, nor is there a pair vertices in $M$ on the same horizontal line as $w_1$ that are to the left and right of $w_1$. Conditions (3) and (6) imply that if there is a vertex in $m_1 \in M$ that is on the same vertical line as any of the vertices on the horizontal line segment between $w$ and $m$, then if $m_1$ is above $w$ and $m$, it must be the rightmost vertex on a horizontal line segment, otherwise it must be the leftmost vertex in a horizontal line segment. This implies that there are no vertices in the north-east or south-west quadrants with respect to $w_1$ as the origin, hence by Lemma~\ref{quadlem2}, $w_1$ does not have a locally resolved neighbourhood.
If $w_1$ were a side vertex, then $w_1 \in M$ and it would either be on the same side and below $u$, where $u$ is a corner vertex, or it would be on the same side and above $v$, where $v$ is a corner vertex. In the case $w_1$ is on the same side as $u$, then there are no other vertices in $M$ on the same horizontal line as $w_1$ and $u$ is the only vertex on the vertical line as $w_1$. Condition (6) implies that there are no vertices in the north-east quadrant with respect to $w_1$ as the origin and hence $w_1$ is not in the situation described by either Side Case (1) or Side Case (2) and thus does not have a locally resolved neighbourhood. In the case $w_1$ is on the same side as $v$, then there are no other vertices in $M$ on the same horizontal line as $w_1$ and $v$ is the only vertex on the vertical line as $w_1$. Condition (6) implies that there are no vertices in the south-west quadrant with respect to $w_1$ and so as in the previous case, $w_1$ does not have a locally resolved neighbourhood.\\
Therefore $M-\{w\}$ is not a resolving set for the grid.
\end{proof}

\begin{figure}[H]
        \centering
        \includegraphics[width=0.5\textwidth]{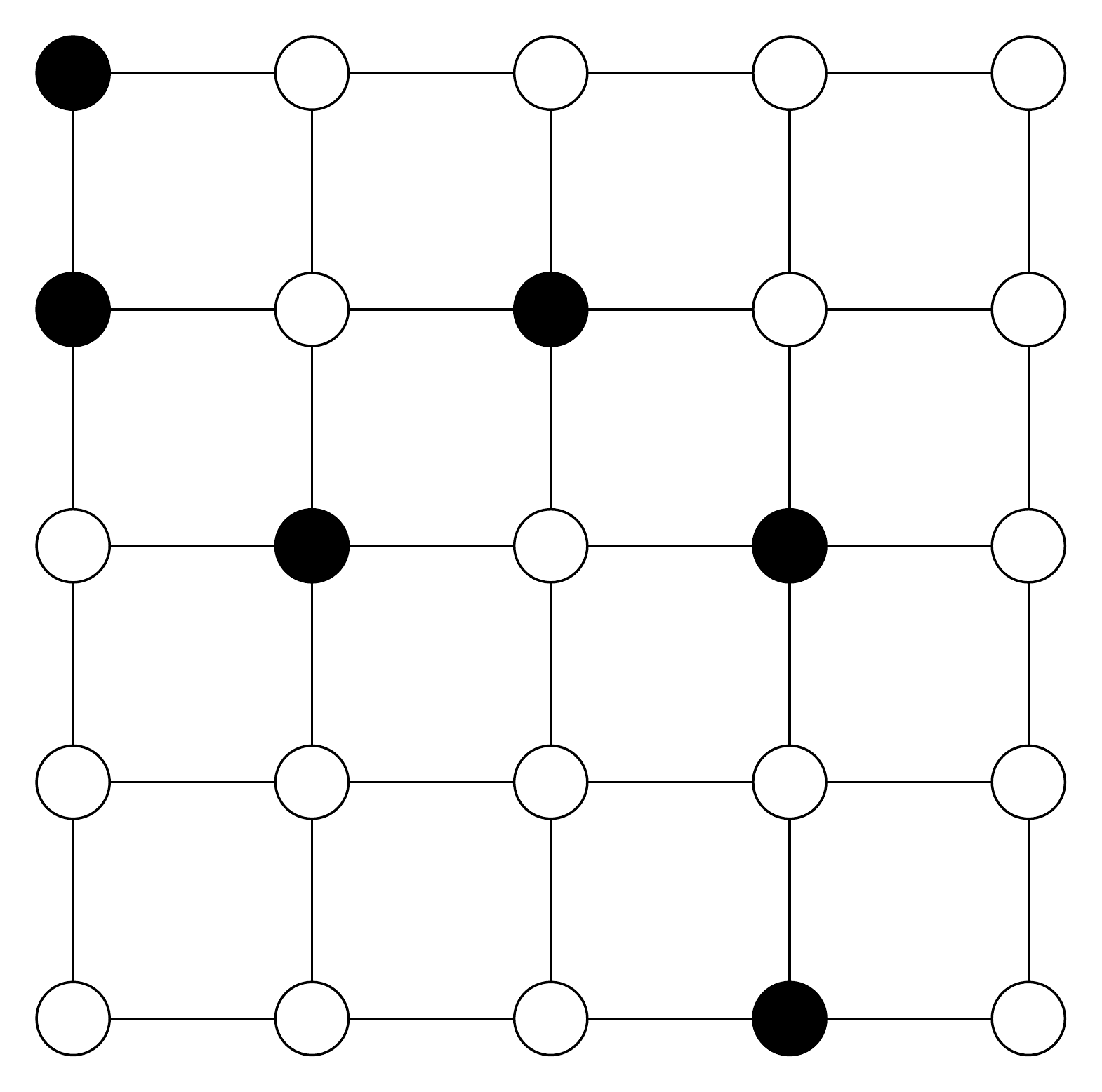}
        \caption{An example of a minimal described by Theorem~\ref{segementthm}.}
\end{figure}

It is important to note that the class of minimals described by Theorem~\ref{segementthm} does not include every $k$-minimal where $k>3$. Below we have an example of a 4-minimal that is clearly not in this class of minimals as no element in the minimal is on the same line as any other element in the minimal. The resolvability and minimality of this set of landmarks was verified by computer using an integer programming formulation of the problem. See appendix~\ref{sec:integer} for details.\\

\begin{figure}[H]
        \centering
        \includegraphics[width=0.5\textwidth]{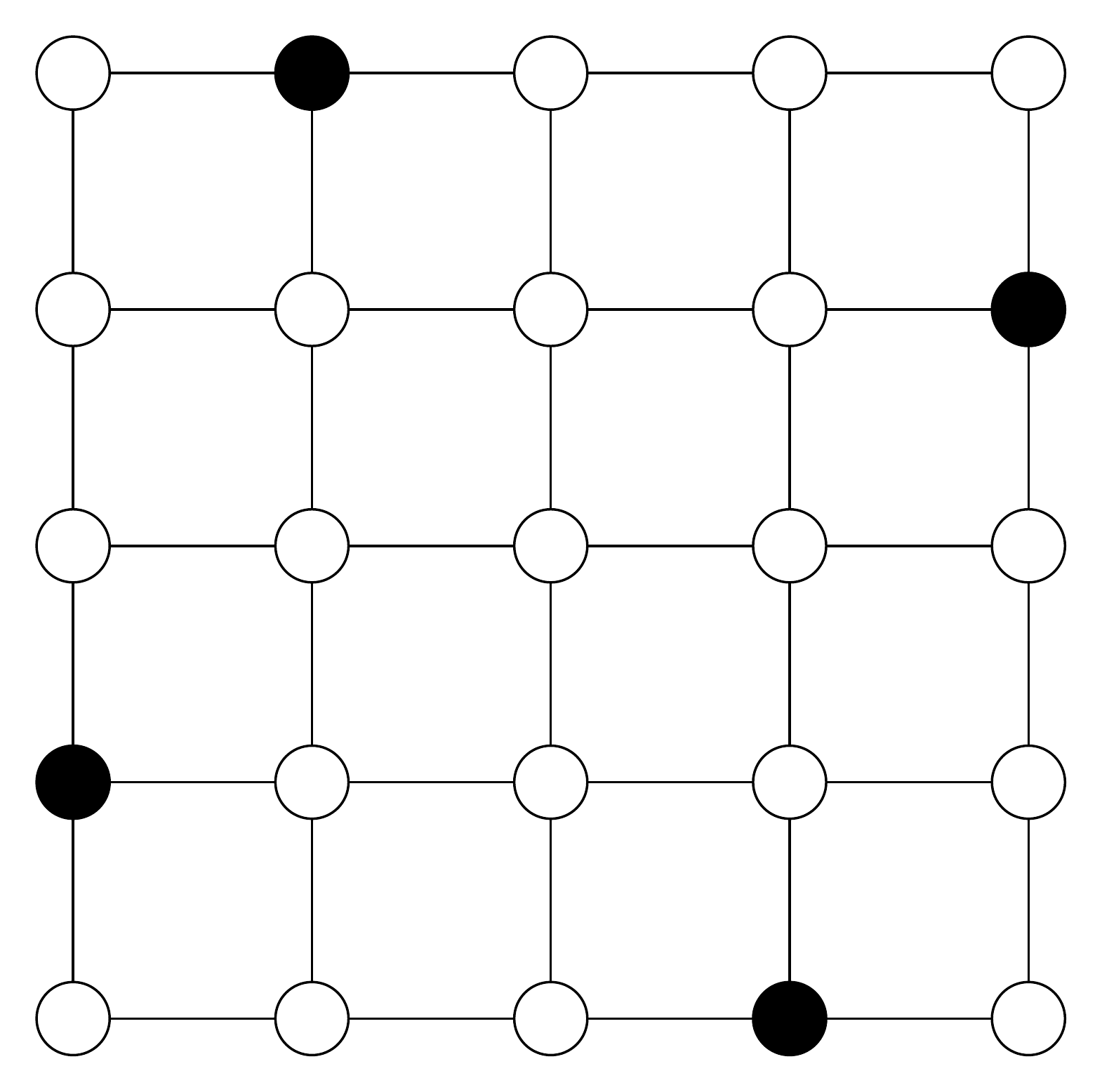}
        \caption{A 4-minimal for the grid $P_5 \square P_5$.}
\end{figure}

We would now like to determine the largest cardinality of any minimal for the grid $P_n \square P_m$, where $3 \leq n \leq m$. The following result can be used to give an upper bound for this number.

\begin{lemma} \label{3line}
No minimal for the grid contains three vertices that are on the same line.
\end{lemma}

\begin{proof}
Suppose we have three distinct vertices $u$, $v$ and $w$ on the same line in the grid, where $w$ is between $u$ and $v$. We will assume without loss of generality that $u$ and $v$ form a vertical line segment and that $u$ is above $v$. Now suppose there are two vertices $p$ and $q$ that are not resolved by $u$ or $v$. There are three possible cases for the location of $p$ and $q$:
\begin{enumerate}[(1)]
\item $p$ and $q$ are on the same horizontal line that intersects the vertical line segment between $u$ and $v$.
\item Both $p$ and $q$ are above $u$.
\item Both $p$ and $q$ are below $v$.
\end{enumerate}
For the first case, let $r$ be the vertex in the vertical line segment between $u$ and $v$ that is on the same horizontal line as $p$ and $q$. Since $r$ does not resolve $p$ or $q$ and since there is a shortest path from $w$ to $p$ that goes through $r$ and a shortest path from $w$ to $q$ that goes through $r$ if $w \neq r$, $w$ does not resolve $p$ and $q$. For the second case, there is a shortest path from $w$ to $p$ that goes through $u$ and a shortest path from $w$ to $q$ that goes through $u$, and so $w$ does not resolve $p$ and $q$. Finally, for the third case  there is a shortest path from $w$ to $p$ that goes through $v$ and a shortest path from $w$ to $q$ that goes through $v$, and thus any pair of vertices $p$ and $q$ that are not resolved by $u$ or $v$, will not be resolved by $w$. Thus, if $u,v$ and $w$ were in a resolving set $R$, $R-\{w\}$ would still be resolving.
\end{proof}

The corollary of this result is that no minimal will have more than $2n$ vertices. It turns that we can provide an even better upper bound.

\begin{theorem}
For the grid $P_n \square P_m$, where $3 \leq n \leq m$, the cardinality of the largest minimal is $2n-2$.
\end{theorem}

\begin{proof}
Suppose we have a grid $P_n \square P_m$, where $3 \leq n \leq m$, and a set $R$ that contains $2n$ vertices, where $R$ does not have three vertices that are on the same line. If we rotate the grid so that the north and south sides are of length $m$, then we have two vertices in every row. This implies that we have two vertices, $u_1$ and $u_2$, on one horizontal side of the grid, and two vertices $v_1$ and $v_2$ that are on the opposite side. If either $u_1$ or $u_2$ are on the same vertical line as any vertex in the horizontal line segment between $v_1$ or $v_2$, or if  $v_1$ or $v_2$ are on the same vertical line as any vertex in the horizontal line segment between $u_1$ or $u_2$, then $R$ is the superset of a 3-minimal and we could remove $2n-3$ vertices from $R$ and still have a resolving set. Assume that neither $u_1$ or $u_2$ are on the same vertical line as a vertex in the line segment between $v_1$ and $v_2$, and assume that neither $v_1$ or $v_2$ are on the same vertical line as a vertex in the line segment between $u_1$ and $u_2$. Hence we will assume without loss of generality that if the horizontal coordinates of $u_1,u_2,v_1$ and $v_2$ are $p_1,p_2,q_1,q_2$, then $p_1<p_2<q_1<q_2$. The vertices $u_1$ and $v_1$ form opposite corners of a subgrid. Let the vertices in this subgrid be in the set $A$, and let all other vertices in the grid be in the set $B$. We will now show that the two vertices $u_2$ and $v_2$ do not affect the resolvability of $R$.\\ Clearly all of the corners of this grid are locally resolved by $u_1$ and $v_1$ since they are on opposite sides of the grid. If a vertex is a side vertex that is not on a side containing $u_1$ or $v_1$ then it is in the situation of Side Case (1) where $u_1$ and $v_1$ are in adjacent quadrants. Any interior vertex in $B$ is in the situation of Interior Case (1) where $u_1$ and $v_1$ are again in adjacent quadrants. A side vertex in $B$ that is on the same side as $u_1$ is in the situation of Side Case (2) where $u_1$ is in a quadrant boundary and $v_1$ is in a quadrant that has the quadrant boundary that contains $u_1$. Similarly, a side vertex in $B$ that is on the same side as $v_1$ is in the situation of Side Case (2) as it is locally resolved by $u_1$ and $v_1$. A side vertex in $A$ that is on the same side as $u$ may not be locally resolved as the quadrant boundary that contains $u_1$ and $u_2$ (or just $u_2$ if $u_1$ is the origin) is not a boundary of the quadrant that contains $v_1$ and $v_2$. If we had an element of $R$ that was in the adjacent quadrant to the quadrant containing $v_1$, then the side vertex would be locally resolved, both with or without $u_2$ and $v_2$. Similarly, the resolvability of the local neighbourhood of a side vertex in $A$ that is on the same side as $v_1$ does not depend on $u_2$ or $v_2$. Finally, an interior vertex in $A$ has the pair $u_1$ and $u_2$ in one quadrant and $v_1$ and $v_2$ in an opposite quadrant with respect to the interior vertex as the origin. It is clear from Lemma~\ref{quadlem2} that the resolvability of the local neighbourhood of this vertex does not depend on $u_2$ or $v_2$. Thus, if $R$ were a resolving set, then $R-\{u_2,v_2\}$ would be a resolving set and therefore no minimal has a cardinality greater than $2n-2$.\\
We show the existence of a minimal of cardinality $2n-2$ by the following construction:
Given the grid $P_n \square P_m$, where $3 \leq n \leq m$, rotate the grid so that the north and south sides are of length $n$. Let $u$ be the north-west corner of the grid and let $v$ be the left neighbour of the south-east corner of the grid. Start from $u=p_1$ and a form path $P= p_1,p_2,p_3,\dots,p_{2n-3}$, where $p_{i+1}$ is the neighbour below $p_i$ if $i$ is odd, else $p_{i+1}$ is the neighbour to the right of $p_i$. Note that the path moves right $n-2$ times, hence $p_{2n-3}$ is on the same vertical line as $v$. Let $M=\{P\} \cup \{v\}$, where $\{ P \}$ denotes the set of vertices in $P$. Then by Theorem~\ref{segementthm}, $M$ is a minimal since the horizontal line segment path from $u$ to $v$ is unique and traverses every vertex in $M$. Since $|M| = 2n-2$, $M$ is a $(2n-2)$-minimal. 
\end{proof}

\begin{figure}[H]
        \centering
        \includegraphics[width=0.5\textwidth]{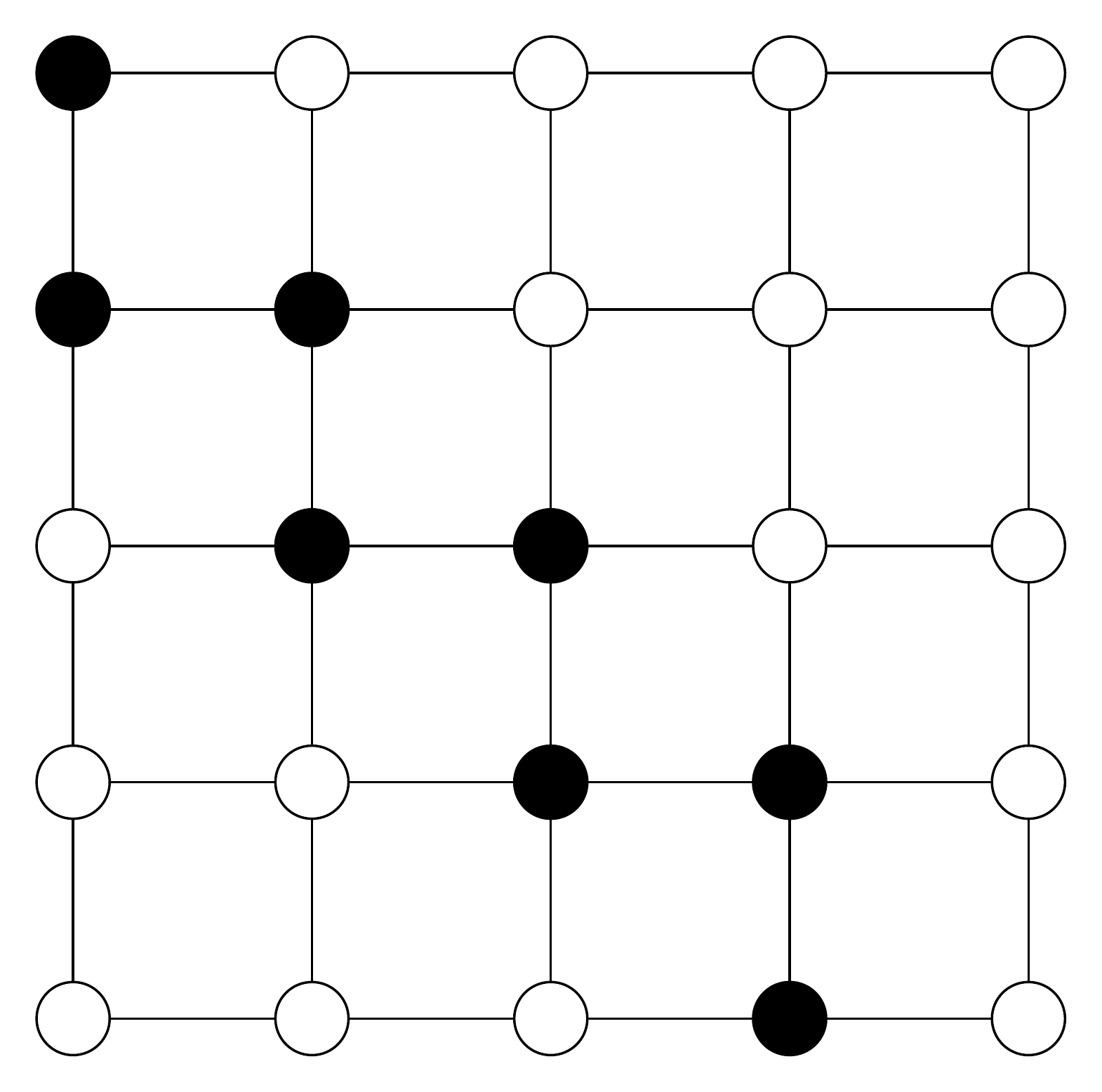}
        \caption{An 8-minimal for the grid $P_5 \square P_5$.}
\end{figure}

The $(2n-2)$-minimal in the above proof was formed by taking a shortest path from two boundary vertices, $u$ and $v$, on opposite sides of the grid and letting the minimal contain $u$,$v$ and the corners of the shortest path, where a corner of the path $v_1,v_2,\dots ,v_k$ is any vertex $v_i$ such that $v_{i-1}$ and $v_{i+1}$ are on different lines of the grid. By choosing an appropriate $u$ and $v$ and an appropriate shortest path between them, it is possible to produce a $k$-minimal in this fashion, for any $k$ such that $3 \leq k \leq 2n-2$. 

\section{Conclusion}
In this paper, we have provided a complete characterisation of 2-minimals and 3-minimals and have shown that $k$-minimals exist if and only if $2 \leq k \leq 2n-2$. We have also provided a characterisation of a class of $k$-minimals and a weak characterisation of resolving sets of the grid which may be used in the future to find more classes of $k$-minimals.  As future work, we wish to give a complete characterisation and enumeration of all the minimals of the grid. The results thus far would suggest that even after a complete characterisation, an enumeration of all the minimals would be difficult, so in the case that no polynomial time algorithm to solve the minimum weight resolving set problem for the grid can be found, we would like to develop a suitable heuristic algorithm for this problem on grid graphs. We would also eventually like to expand the scope of our investigation to include other grid-like graphs, namely cylinders ($P_n \square C_m$) and toruses ($C_n \square C_m$). In the case of the torus, we conjecture that all the minimals for vertex transitive graphs are of the same cardinality which, if true, would imply that no investigation into the torus is needed since it is vertex transitive and its metric dimension is known. Proving this conjecture is also possible future work and due to the vast number of graphs for which metric dimension is known, this would be a very powerful result. 

\bibliographystyle{apalike}
\bibliography{ref}

\appendix
\appendixpage
\section{Integer Programming Formulation}\label{sec:integer}
This formulation is a modification of the integer programming formulation developed by Chartrand et al. \citep{chartrand2000} to solve metric dimension problems, where we have included weights in the objective. Given a graph $G$, let $V = \{v_1,v_2, \dots ,v_n\}$ be the set of vertices and let $\{w_1,w_2,\dots w_n\}$ be the corresponding weights. For each vertex $v_i$, let $x_i$ be the binary decision variable that is 1 if the vertex $v_i$ is in the resolving set and 0 otherwise. We assume that the distances  $d(v_i,v_j)$ for each pair $v_i,v_j \in V$ have been precomputed, which can be done in polynomial time using an appropriate algorithm such as Floyd's algorithm \cite{floyd}. Our formulation is as follows.

\begin{alignat*}{2}
\text{min}\ & \sum_{i = 1}^n  w_i x_i \\
\text{s.t.}\ & \sum_{k =1}^n |d(v_i,v_k) - d(v_j,v_k)|x_k > 0  \quad  \forall 1 \leq i <j \leq n \\
& x_i\ \text{integer} \quad  \forall i \in \{1,2 \dots n\} 
\end{alignat*}

Note that if all vertices have unit weight, then our formulation will find a metric basis as the solution. In order to use the integer program to see if a set of vertices $\{u_1,u_2, \dots u_m \} \subseteq V$ is a minimal, we give each vertex in our set weight 1, and we give every other vertex weight $M$, where $M>m$. If the integer program gives an objective value of $m$, then we know our set is a minimal. If the objective value is less than $m$, then our set is resolving but not minimal. If the objective value is greater than $m$, then our set is not resolving.

\end{document}